\documentclass{amsart}
\usepackage{amsmath}
\usepackage{amscd}
\usepackage{amsxtra}
\usepackage{amssymb}
\usepackage[inline]{enumitem}
\usepackage{verbatim}
\usepackage[mathscr]{eucal}	

\DeclareSymbolFont{pssymbols}     {OMS}{ztmcm}{m}{n}
\DeclareSymbolFontAlphabet{\mathpsscr}   {pssymbols}

\newcommand{\mycomment}[1]{%
}%

\theoremstyle{plain}
\newtheorem{thm}[subsection]{Theorem}
\newtheorem{cor}[subsection]{Corollary}
\newtheorem{prop}[subsection]{Proposition}
\newtheorem{lem}[subsection]{Lemma}

\newtheorem*{thm*}{Theorem}
\newtheorem*{cor*}{Corollary}
\newtheorem*{prop*}{Proposition}
\newtheorem*{lem*}{Lemma}

\theoremstyle{definition}
\newtheorem{defn}[subsection]{Definition}

\theoremstyle{remark}
\newtheorem{rem}[subsection]{Remark}

\newtheorem*{rem*}{Remark}
\newtheorem*{rems*}{Remarks}

\newtheorem*{note*}{Note}

\setlist[enumerate,1]{label=\textup{(\roman*)}}
\setlist[enumerate]{font=\normalfont}
\newlength{\myleftmargin}
\setlist[enumerate]{leftmargin=\myleftmargin}

\numberwithin{equation}{section}

\usepackage[arrow,matrix,tips,frame,curve,ps]{xy}
\SelectTips{cm}{}
\UseTips
\CompileMatrices

\newdir^{ (}{{}*!/-5pt/@^{(}}
\newdir_{ (}{{}*!/-5pt/@_{(}}

\newcommand{\CC}{{\mathbb C}}
\newcommand{\RR}{{\mathbb R}}
\newcommand{\ZZ}{{\mathbb Z}}

\newcommand{\QQ}{{\mathbb Q}}
\newcommand{\EE}{{\mathbb E}}

\newcommand{\VV}{{\mathbb V}}

\DeclareMathOperator{\codim}{codim}

\DeclareMathOperator{\Int}{int}
\DeclareMathOperator{\Ker}{Ker}
\renewcommand{\ker}{\Ker}
\renewcommand{\Im}{\operatorname{Im}}

\DeclareMathOperator{\cone}{M}
\DeclareMathOperator{\mS}{SS}
\DeclareMathOperator{\wmS}{SS_w}
\DeclareMathOperator{\muwmS}{SS_{w,\mu}}
\DeclareMathOperator{\nuwmS}{SS_{w,\nu}}
\DeclareMathOperator{\etawmS}{SS_{w,\eta}}

\DeclareMathOperator{\type}{Type}

\renewcommand{\l}{\ell}

\newcommand{\rvvv}[1][]{\rVert_{\if!#1!b\else#1,b\fi}}	

\newcommand{\lsp}[1]{{}^{#1}\!}

\newbox\arrowbox
\setbox\arrowbox=\hbox{\lower .5ex\rlap{$\longrightarrow$}%
 \raise .5ex\hbox{$\longleftarrow$}}

\newcommand{\str}{\textup{str}}
\newcommand{\LprecLeft}{\mathrel{\leqslant_\mu}}
\newcommand{\LprecRight}{\mathrel{\leqslant}_\nu}
\newcommand{\LprecEta}{\mathrel{\leqslant}_\eta}

\newcommand{\G}{\Gamma}

\DeclareMathOperator{\GL}{GL}

\newcommand{\Pl}{\mathscr P}	

\newcommand{\Xhat}{\widehat{X}}

\newcommand{\X}{\mathcal X}
\newcommand{\Char}{X}

\renewcommand{\i}{i}		
\newcommand{\ihat}{{\hat{\imath}}}

\newcommand{\jhat}{{\hat{\jmath}}}
\DeclareMathOperator{\id}{id}
\newcommand{\LeftNaturalMorphism}{\kappa}
\newcommand{\RightNaturalMorphism}{\sigma}
\newcommand{\LeftMappingConeMorphism}{\alpha}
\newcommand{\RightMappingConeMorphism}{\beta}

\newcommand{\C}{{\mathcal C}}
\newcommand{\I}{\mathcal I}

\let\Sh\Sheaf

\newcommand{\WnC}{\mathcal W^\eta \mathcal C}
\newcommand{\W}{{\mathcal W}}

\renewcommand{\L}{\mathscr L}
\newcommand{\M}{\mathcal M}
\newcommand{\N}{\mathcal N}
\renewcommand{\P}{\mathcal P}

\newcommand{\sa}{{\mathfrak a}}

\newcommand{\n}{{\mathfrak n}}

\newcommand{\al}{\alpha}

\renewcommand{\b}{\beta}
\newcommand{\hsr}{\rho}

\newcommand{\D}{\Delta}
\newcommand{\Dhat}{{\widehat{\D}}}

\renewcommand{\t}{\tau}
\newcommand{\U}{\varUpsilon}

\makeatletter
\def\sphat{^{\mathchoice{}{}%
 {\,\,\smash[b]{\hbox{\lower4\ex@\hbox{$\m@th\widehat{\null}$}}}}%
 {\,\smash[b]{\hbox{\lower3\ex@\hbox{$\m@th\hat{\null}$}}}}}\,}
\makeatother

\newcommand{\Mor}{\operatorname{Mor}}

\newcommand{\IrrRep}{\operatorname{\mathfrak I\mathfrak r\mathfrak r}}

\let\Der\Derived

\newcommand{\Complex}{\operatorname{\mathbf C}}
\newcommand{\K}{\operatorname{\mathbf K}}

\newcommand{\Mod}{\operatorname{\mathbf M \mathbf o \mathbf d}}
\newcommand{\Graded}{\operatorname{\mathbf G\mathbf r}}

\makeatletter
\def\prime{{\null\prime@\null}}
\mathchardef\prime@="0230
\makeatother

\begin{document}


\author{Leslie Saper}
\address{Department of Mathematics\\ Duke University\\ Box 90320\\ Durham,
NC 27708\\U.S.A.}
\email{saper@math.duke.edu}
\urladdr{http://www.math.duke.edu/faculty/saper}
\title{$\mathscr L$-modules are mixed}
\begin{abstract}
  Let $X$ be the locally symmetric space associated to a reductive $\QQ$-group $G$ and an arithmetic subgroup $\Gamma$.  An $\L$-module $\M$ is a combinatorial model of a constructible complex of sheaves on $\Xhat$, the reductive Borel-Serre compactification of $X$ whose strata $X_P$ are indexed by $\G$-conjugacy classes of parabolic $\QQ$-subgroups $P$ of $G$.  Important cohomology theories on $\Xhat$ such as ordinary cohomology, weighted cohomology, and intersection cohomology can be realized as the cohomology of $\L$-modules.  We show that any $\L$-module $\M$ is ``mixed'' in the sense it is an iterated mapping cone of maps to or from weighted cohomology $\L$-modules $\ihat_{P*}\WnC(V)[-d]$ on strata $X_P$ of $\Xhat$; here $\eta$ is a middle weight profile and $V$ is an irreducible regular $L_P$-module.  These weighted cohomology ``building blocks'' are indexed (up to multiplicity) by $V\in\wmS(\M)$, the weak micro-support which is a computable local invariant of $\M$.  As an application we prove that the intersection cohomology of $\Xhat$ is isomorphic to the weighted cohomology of $\Xhat$, at least excluding $\QQ$-types $D$, $E$, and $F$.
\end{abstract}
\maketitle



\setcounter{tocdepth}{1}
\tableofcontents


\section{Introduction}
\label{sectIntro}

\subsection{The setting}
Let $X$ be the locally symmetric space associated to a reductive group $G$ over $\QQ$ and an arithmetic group $\Gamma$.  We will work with the reductive Borel-Serre compactification $\Xhat$ of $X$ \cite{refnZuckerWarped,refnGoreskyHarderMacPherson}.  Its strata $X_P$ are indexed by the $\Gamma$-conjugacy classes of parabolic $\QQ$-subgroups $P$ of $G$.%
\footnote{We use $P$ to denote both a parabolic $\QQ$-subgroup and its $\Gamma$-conjugacy class.}
For such a $P$ let $N_P$ be its unipotent radical and let $L_P=P/N_P$  be its Levi quotient.  The stratum $X_P$ is the locally symmetric space associated to $L_P$ and  its induced arithmetic subgroup $\G_{L_P}$.  The closure $\Xhat_Q$ of a stratum $X_Q$ contains all strata $X_P$ for which $P\subseteq Q$.
Note that $\Xhat$ may have odd codimension strata even if $X$ is Hermitian (which we do not assume).

Many important cohomology theories can be realized as the hypercohomology of a constructible complex of sheaves on $\Xhat$.  For example, there is the ordinary cohomology $H(X;\EE)$, the intersection cohomology $I_pH(\Xhat;\EE)$ \cite{refnGoreskyMacPhersonIHTwo}, and the weighted cohomology $W^\eta H(\Xhat;\EE)$ \cite{refnGoreskyHarderMacPherson}.  (Here $\EE$ is the local coefficient system on $X$ associated to a regular representation $G\to\GL(E)$.)

\subsection{$\L$-modules}
An $\L$-module \cite{refnSaperLModules,refnSaperIHP,refnSaperCDMreprint} is a combinatorial analogue of a constructible complex of sheaves on $\Xhat$.  Specifically let $W\subseteq \Xhat$ be a locally closed union of strata with a unique maximal stratum $X_S$ and let $\Pl(W)$ be the set of $\Gamma$-conjugacy classes of parabolic $\QQ$-subgroups indexing the strata of $W$.  An $\L$-module $\M$ on $W$ is given by the data of a graded regular $L_P$-module $E_P$ for every $P\in \Pl(W)$ and  morphisms $f_{PQ}\colon H(\n_P^Q;E_Q)\to E_P[1]$ for every $P\le Q$; this data must satisfy certain differential-like conditions.  For any $\L$-module $\M$ there is a realization $\Sh(\M)$ as a constructible complex of sheaves on $\Xhat$.   The global cohomology $H(\Xhat;\M)$ of an $\L$-module $\M$ is defined to be the hypercohomology of $\Sh(\M)$.

Each of the cohomology theories mentioned above can be realized by $\L$-modules.  Namely the cohomology of the $\L$-module $i_{G*}E$ is $H(X;\EE)$ \cite[\S11]{refnSaperIHP}, there is an $\L$-module $\WnC(E)$ whose cohomology is $W^\eta H(\Xhat;\EE)$ \cite[\S6]{refnSaperLModules}, and there is an $\L$-module  $\I_p\C(E)$ whose cohomology is intersection cohomology $I_pH(\Xhat;\EE)$ \cite[\S5]{refnSaperLModules}.

We note two advantages in working with $\L$-modules as opposed to complexes of sheaves.  First, many of the usual functors on the derived category of sheaves can be directly realized on $\L$-modules.  For example, if $\M$ is an $\L$-module on $W$ and $i_{P}\colon X_P\hookrightarrow W$ is the inclusion of a stratum then the $\L$-module $i_P^!\M$ on $X_P$ is the complex $(E_P,f_{PP})$.  And if $k\colon W\hookrightarrow W'$ is an inclusion then the $\L$-module $k_*\M$ is defined by extending the data of $\M$ by $0$ for $P\in\Pl(W')\setminus \Pl(W)$.  A second advantage is that both the local cohomology $H(i_P^*\M)$ and the local cohomology with supports $H(i_P^!\M)$ are endowed with the structure of a representation of $L_P$, not just of $\Gamma_{L_P}$.

\subsection{Micro-support}
The most significant advantage of an $\L$-module $\M$ is that one can define its \emph{micro-support} $\mS(\M)$, a particular set of irreducible regular representations $V$ of the various $L_P$.
To describe it, decompose $L_P = M_P\cdot S_P$ where $S_P$ is the maximal $\QQ$-split torus in the center of $L_P$.
Then $V\in\mS(\M)$ if (a) $V|_{M_P}$ is conjugate self-contragradient and (b) there is some $Q\supseteq P$ such that the local cohomology of $\M$ on $X_P$ supported on $X_Q$ is nonzero and the dominant cone of $S_P/S_Q$ acts nonpositively on $V\otimes\hsr_P$.  We denote this local cohomology $\type_{Q,V}(\M)$.

An important property of $\mS(\M)$ is that it controls the global cohomology of  $\M$ in the sense that $H(\Xhat;\M)=0$ if $\mS(\M)=\emptyset$.  We will see a more subtle one below.

\subsection{Building blocks}
Let $\eta=\mu$ or $\nu$, the two middle weight profiles for weighted cohomology.  We define a variant of micro-support, $\mS_\eta(\M)$, by making a special choice above for $Q$ depending on $\eta$ (see \S\ref{ssectMicroSupport}) and likewise define $\type_{\eta,V}(\M)$.  The set $\mS_\eta(\M)$ again controls the global cohomology in the sense above and, in addition, one finds that the weighted cohomology $\L$-module $\ihat_{P*}\WnC(V)$ has $\mS_\eta(\ihat_{P*}\WnC(V))=\{V\}$ provided $V|_{M_P}$ is conjugate self-contragradient.  This suggests, at least from the point of view of cohomology, that the $\L$-modules $\ihat_{P*}\WnC(V)$ for all $P\in\Pl(W)$ and all $L_P$-modules $V\in\mS_\eta(\M)$ could serve as ``building blocks'' to understand a general $\L$-module $\M$.

To implement this one might try to find for an $L_P$-module $V\in\mS_\eta(\M)$ a morphism $\phi\colon\ihat_{P!}\WnC(V)[-d]\to \M$ or $\psi\colon \M \to \ihat_{P*}\WnC(V)[-d]$ that yield a nonzero map on $\type_{\eta,V}$. Then in the homotopy category the mapping cone $\widetilde \M$ would have smaller $\type_{\eta,V}$ and we could repeat the process with some $\widetilde V\in\mS_\eta(\widetilde\M)$.  Actually to make this work one needs to also consider $V$ which fail the conjugate self-contragradient condition.  Thus we work with the potentially larger \emph{weak $\eta$-micro-support} $\etawmS(\M)$ in which that condition is dropped.  Note that the additional building blocks we now consider, $\ihat_{P*}\WnC(V)$ for $V\in \etawmS(\M) \setminus \mS_\eta(\M)$, have zero global cohomology since $\mS_\eta(\ihat_{P*}\WnC(V))=\emptyset$ when $V|_{M_P}$ is not conjugate self-contragradient.

\subsection{A partial order}
For this process to succeed it is essential to deal with the elements of $\etawmS(\M)$ in the correct order.  In \S\ref{sectPartialOrder} we define a partial order $\LprecEta$ on the elements of $\etawmS(\M)$.  In the case $\eta=\mu$, if $V$ and $\widetilde V$ are irreducible $L_P$- and $L_{\widetilde P}$-modules respectively and if $P\le \widetilde P$ then $V \LprecLeft \tilde V$ if $(\xi_V+\hsr_P)|_{\sa_P^{\widetilde P}} \in  -\lsp+\sa_P^{ \widetilde P*}$ (the negative root cone)  while $\tilde V \LprecLeft V$ if $(\xi_V+\hsr_P)|_{\sa_P^{\widetilde P}} \in  \Int \lsp+\sa_P^{ \widetilde P*}$.  One must then compose such relations to obtain a transitive relation.  The case $\eta=\nu$ has the interior condition moved to the negative cone.

\subsection{Bounded $\L$-modules are mixed}
The main result of this paper, Theorem ~\ref{thmBoundedLModulesAreMixed}, is that any bounded $\L$-module $\M$ can be realized in the homotopy category as an iterated mapping cone of morphisms with weighted cohomology.  Specifically when $\eta=\mu$ there is a sequence of $\L$-modules $\M=\M_N,\M_{N-1},\dots,\M_0=0$ such that $\M_{i-1}$ is the mapping cone of a morphism $\phi_i\colon \ihat_{P*}\W^\mu\C(V_i)[-d_i]\to \M_i$; the $V_i$ are nondecreasing with respect to $\LprecLeft$.  For $\eta=\nu$ the same holds but with morphisms $\psi_i\colon \M_i[1]\to \ihat_{P*}\W^\nu\C(V_i)[-d_i+1]$ and the $V_i$ nonincreasing with respect to $\LprecRight$.  We refer to this property of a bounded $\L$-module $\M$ as being \emph{$\eta$-mixed}.

\subsection{An application}
As an initial application of this result we prove in Theorem~ \ref{thmIsomorphismIntersectionAndWeightedCohomology} that if $E$ is conjugate self-contragredient then $W^\eta H(\Xhat;\EE)\cong I_pH(\Xhat;\EE)$ provided the $\QQ$-root system of $G$ does not involve types $D$, $E$, and $F$.%
\footnote{The assumption on the $\QQ$-root system is removable if the calculation of the micro-support of intersection cohomology in \cite[\S17]{refnSaperLModules} can be generalized as is expected.}
Here $\eta$ is a middle weight profile and $p$ is the corresponding middle perversity.
In the special case when $X$ is Hermitian (and in some equal-rank settings) such global isomorphisms could be obtained by combining the main results of \cite{refnGoreskyHarderMacPherson} and of \cite{refnSaperLModules}.  That proof involves showing that  the sheaves had isomorphic local cohomology after taking the push forward to the Baily-Borel Satake compactification $X^*$.  The current proof on the other hand does not involve an isomorphism of local cohomology on any space.  It is an interesting problem to determine a compactification of $X$ (aside from the one point compactification) such that the push forward of the sheaves have isomorphic local cohomology.

\subsection{Future plans}
In future work we plan to use micro-support and weighted cohomology ``building blocks'' in order to describe the ordinary cohomology $H(X;\EE)$.  The isomorphism between weighted cohomology and intersection cohomology noted above indicates such a description would be topological in nature.   On the other hand we also plan to show that $\Im (W^\mu H(\Xhat;\EE) \to W^\nu H(\Xhat;\EE))$ is isomorphic to the space of $L^2$-harmonic $\EE$-valued forms on $X$.  Thus this work should be related to the description of ordinary cohomology via cusp forms, Eisenstein series, and residues of Eisenstein series as in the work of Langlands \cite{refnLanglandsFE}, Harder \cite{refnHarderQRankOne,refnHarderGLtwo}, Schwermer \cite{refnSchwermerKohomologie,refnSchwermerGeneric}, Franke \cite{refnFranke}, Franke and Schwermer \cite{refnFrankeSchwermer}, Li and Schwermer \cite{refnLiSchwermer}, and many others.

\subsection{Acknowledgments}
The ideas here have been percolating for some time.  I would like to Mark Goresky, Michael Harris, G\"unther Harder, Mike Lipnowski, Eduard Looijenga, Amnon Neeman, Bill Pardon, Birgit Speh, and Steve Zucker for helpful and stimulating conversations.  I would also like to thank the \textit{\'Equipes Formes Automorphes} at the \textit{Institut de Mathématiques de Jussieu} for their hospitality while some of this work was performed.

\section{Notation}

\subsection{}
\label{ssectNotationReductiveGroup}
For any algebraic group $G$ defined over $\RR$ we denote the Lie algebra of its real points by the corresponding Fraktur letter, $\mathfrak g=\operatorname{Lie} G(\RR)$.  If $G$ is defined over $\QQ$ we let $\Char(G)$ denote the characters of $G$ defined over $\QQ$.  For $\psi\in\Char(G)$ we also let $\psi$ denote the induced element of $\mathfrak g^*$.

Let $G$ be a connected reductive $\QQ$-group and $\G$ an arithmetic subgroup. Let $S_G$ be the maximal $\QQ$-split torus in the center of $G$ and let $A_G=S_G(\RR)^0$.  There is an almost direct product decomposition $G=M_G\cdot S_G$ where $M_G=\bigcap_{\chi\in\Char(G)}\ker \chi^2$.  Let $K \subseteq G(\RR)$ be a maximal compact subgroup. The associated locally symmetric space is $X = \G\backslash G(\RR)/KA_G$.

\subsection{}
Let $\Pl$ denote the finite set of $\G$-conjugacy classes of parabolic $\QQ$-subgroups of $G$; we do not distingish notationally a parabolic subgroup from its conjugacy class.  Inclusion of $\G$-conjugacy classes induces a partial order on $\Pl$ which we denote $\le$.  For example, $P<Q$ means $P\subsetneq Q^\gamma$ for some $\gamma \in \G$.  For $P\le R\in\Pl$ let $[P,R]\subseteq \Pl$ denote the closed interval consisting of $Q$ such that $P\le Q\le R$.

\subsection{}
If $P$ is a parabolic $\QQ$-subgroup of $G$ with unipotent radical $N_P$, let  $L_P=P/N_P$ be its Levi quotient.  Let $S_P$ be the maximal $\QQ$-split torus in the center of $L_P$; there is an almost direct product decomposition $L_P=M_P \cdot S_P$.  For $P\le R$ let $S_P^R=\bigl(\bigcap_{\lambda\in \Char(L_R)}\ker\lambda\cap S_P\bigr)^0$ and $N_P^R = N_P/N_R$.  We have an almost direct product $S_P=S_R\cdot S_P^R$.  If we set $A_P^R=S_P^R(\RR)^0$ and similarly for $A_P$ we have a direct product $A_P=A_R\times A_P^R$.  This induces $\sa_P^* = \sa_R^*+\sa_P^{R*}$ and for $\lambda\in\sa_P^*$ we write correspondingly $\lambda=\lambda_R + \lambda_P^R$.

For a parabolic $\QQ$-subgroup $P$ we extend the notation of \S\ref{ssectNotationReductiveGroup} to the Levi quotient $L_P$ with its induced arithmetic subgroup $\G_{L_P} = \G/(\G\cap N_P)$.  We obtain the locally symmetric space $X_P = \G_{L_P}\backslash L_P(\RR)/K_PA_P$.

\subsection{}
For any connected reductive $\QQ$-group $L$ we define $\Mod(L)$ to be the category of regular representations of $L$, $\Graded(L)$ the category of graded objects of $\Mod(L)$, and $\Complex(L)$ the category of complexes of objects.  If $C$ is an object of $\Graded(L)$ or $\Complex(L)$ define its shift by $k$ to be $C[k]^i=C^{i+k}$ and in the case of $\Complex(L)$ also multiply the differential by $(-1)^k$.  For any functor from $\Mod(L)$ to an additive category we implicitly extend it to the categories of graded objects or complexes.

A object $(C^i)_{i\in\ZZ}$ in $\Graded(L)$ is called \emph{bounded} if there exists $N\in\ZZ$ such that $C^i=0$ for $|i|>N$; the same definition applies to complexes.  We denote the full subcategories consisting of bounded objects as $\Graded^b(L)$ and $\Complex^b(L)$.%
\footnote{As in \cite[\S\S12.13, 12.16]{refnstacks-project} we view an object of $\Graded(L)$ or  $\Complex(L)$ as a family $(C^i)_{i\in\ZZ}$ of objects of $\Mod(L)$ (together with differential morphisms in the case $\Complex(L)$).  In particular while $C^i$ is finite dimensional for all $i$ (being a regular representation) the direct sum $\bigoplus_i C^i$ may not be finite dimensional.  For $\Graded^b(L)$ or  $\Complex^b(L)$ however $\bigoplus_i C^i$ is a regular represention.
}

Let $\IrrRep(L)$ denote the set of irreducible objects of $\Mod(L)$.  Elements of $\IrrRep(L)$ are usually denoted $V$ but with subscripts or decorations to distinguish different representations.  If $H$ is an object of $\Mod(L)$, $\Graded(L)$, or $\Complex(L)$ and $V\in \IrrRep(L)$ we let $H_V$ be its $V$-isotypic component.

If $L=L_P$ for $P\in\Pl$ we often write $V_P$ for an element of $\IrrRep(L_P)$.  Let $\xi_{V_P}$ denote the character by which $S_P$ acts on $V_P$ as well as the induced element of $\sa_P^*$.  Let $V_{P,\lambda}$ denote the irreducible $L_P$-module with highest weight $\lambda$.  

\subsection{}
If $P_0$ is a minimal parabolic $\QQ$-subgroup of $G$ then $S_{P_0}$ is a maximal $\QQ$-split torus and there is a unique ordering on the $\QQ$-root system of $G$ for which the roots appearing in $\n$ are positive.  Let $\D$ denote the corresponding simple roots.
As usual let $\hsr\in\sa^{G*}$ be $1/2$ the sum of the positive $\QQ$-roots.

For $P\ge P_0$ let $\D^P$ be the simple roots in $L_P$ and let $\D_P$ be the restrictions of $\D\setminus \D^P$ to $S_P$.  If $R \ge P$ let $\D_P^R$ be the subset of elements of $\D_P$ that restrict to $1$ on $S_R$.  Note that $R \longleftrightarrow \D_P^R$ is a one-to-one correspondence between parabolic $\QQ$-subgroups containing $P$ and subsets of $\D_P$.  If $Q,R\ge P$ let $Q\vee R$ be the parabolic $\QQ$-subgroup corresponding to $\D_P^Q\cup\D_P^R$.
Let $(P,R)\ge P$ be the parabolic $\QQ$-subgroup corresponding to  $\D_P\setminus \D_P^R \subseteq \D_P$; it satisfies $(P,R)\cap R = P$ and $(P,R)\vee R=G$.

For $\al\in \D$ let $\al\spcheck\in\sa^G$ be the corresponding coroot.  Let $\Dhat=\{\b_\al\}$ be the basis of $\sa^{G*}$ dual to the coroots.  More generally define the coroot basis $\{\al\spcheck\}_{\al\in\D_P^R}$ of $\sa_P^R$ as in \cite{refnArthurTraceFormula} and let $\Dhat_P^R = \{\b_\al^R\}_{\al\in \D_P^R}$ be the dual basis of $\sa_P^{R*}$.  On the other hand we have $\D_P^R$ as a basis of  $\sa_P^{R*}$ and we let $\{\b\spcheck\}_{\b\in \Dhat_P^R}$ be the dual basis of $\sa_P^R$.

In $\sa_P^{R*}$ we define the \emph{root cone} and the \emph{dominant cone} as
\begin{equation}\label{eqnRootWeightCone}
  \begin{split}
    \lsp{+}\sa_P^{R*}&=\{\,\lambda\in\sa_P^{R*}\mid \langle \lambda,\b\spcheck\rangle\ge0 \text{ for all $\b\in\Dhat_P^R$}\,\} \text{ and}\\
    \sa_P^{R*+}&=\{\,\lambda\in\sa_P^{R*}\mid \langle \lambda,\al\spcheck\rangle\ge0 \text{ for all $\al\in\D_P^R$}\}
  \end{split}
\end{equation}
respectively even though $\D_P^R$ may not be the basis of a root system.

\subsection{}
For $P\le R$ define a partial order on $\sa_P^{R*}$ by setting
\begin{align}
\label{eqnNewGE}
\xi\ge \xi' \quad&\Longleftrightarrow \quad \xi-\xi'\in \lsp{+}\sa_P^{R*}\ ; 
\\
\intertext{if $\xi\ge\xi'$ and $\xi\neq\xi'$ we write $\xi \gneq\xi'$.  Finally define}
\label{eqnNewLT}
\xi > \xi' \quad&\Longleftrightarrow \quad  \xi-\xi'\in \Int\lsp{+}\sa_P^{R*}  
\ .
\end{align}

\subsection{}
Let $W$ be the Weyl group of the $\CC$-root system of $G$; let $\l(w)$ denote the length of $w\in W$.  For a parabolic $\QQ$-subgroup $Q$ of $G$ let $W^Q\subseteq W$ denote the Weyl subgroup for the $\CC$-root system of $L_Q$ and let  $W_Q\subseteq W$ denote the set of minimal length representatives of $W^Q\backslash W$.  If $w\in W$ we write $w=w^Qw_Q$ according to $W=W^QW_Q$.  If $P$ is a parabolic $\QQ$-subgroup of $Q$ let $W_P^Q\subseteq W^Q$ be the set of minimal length representatives of $W^P\backslash W^Q$ so $W^Q=W^PW_P^Q$.

\subsection{The reductive Borel-Serre compactification}
The reductive Borel-Serre compactification \cite{refnGoreskyHarderMacPherson} of $X$ is denoted $\Xhat$; it was first used by Zucker in \cite{refnZuckerWarped}.  Its strata $X_P$ are indexed by $\Pl$.  The closure of $X_P$ in $\Xhat$ is the reductive Borel-Serre comapctifcation of $X_P$ and is denoted $\Xhat_P$; the open star neighborhood of $X_P$ is $U_P=\bigcup_{R\ge P} X_R$.  A union of strata $W\subseteq \Xhat$ is called \emph{admissible} if it is locally closed; let $\Pl(W) \subseteq \Pl$ be the subset indexing the strata of $W$.  Note the locally closed condition is equivalent to  $[P,R]\subseteq \Pl(W)$ for all $P\le R\in\Pl(W)$

For $P\in\Pl(W)$ we will use the following inclusions of admissible subsets of $W$:
\begin{alignat*}{2}
  i_P&\colon X_P \hookrightarrow W\ ,\qquad & j_P &\colon(U_P\setminus X_P)\cap W \hookrightarrow W\ ,\\
  \ihat_P&\colon \Xhat_P\cap W \hookrightarrow W, \qquad & \jhat_P&\colon W\setminus (\Xhat_P\cap W)\hookrightarrow W\ .
\end{alignat*}
Thess maps depend on $W$ but we supress it from the notation.  But note that when these maps (or the functors to be associated to them) are composed $W$ will change at each step.

\section{$\L$-modules and micro-support}
We will briefly define $\L$-modules, their micro-support, and properties of them that we will need.  References are \cite{refnSaperLModules}, \cite{refnSaperIHP}, \cite{refnSaperLtwoLmoduleI}, and the more expository \cite{refnSaperCDMreprint}.

\subsection{Kostant's theorem}
\label{ssectKostantsTheorem}
If $P\le Q\in\Pl$ and $V_{Q,\lambda}\in\IrrRep(L_Q)$ then the Lie algebra cohomology $H(\n_P^Q;V_{Q,\lambda})$ is a representation of $L_P$.  Kostant's theorem \cite{refnKostant} says that
\begin{equation}\label{eqnKostant}
  H(\n_P^Q;V_{Q,\lambda}) = \bigoplus_{w\in W_P^Q}  H(\n_P^Q;V_{Q,\lambda})_w =  \bigoplus_{w\in W_P^Q} V_{P,w(\lambda_Q+\hsr_Q)-\hsr_Q}[-\l(w)] \ .
\end{equation}
The theorem implies \cite{refnGoreskyHarderMacPherson,refnSchwermerGeneric} that if $P\le Q\le R\in\Pl$ and $V_R\in\IrrRep(L_R)$ that
\begin{equation}
H(\n_P^R;V_R) \cong H(\n_P^Q;H(\n_Q^R;V_R))\ .
\end{equation}
To see this one checks that $W_P^R=W_P^Q W_Q^R$ and, if $w= w^Q w_Q\in W_P^R$, that $\l(w)=\l(w^Q)+\l(w_Q)$.

\subsection{Definition of $\L$-modules}
\label{ssectDefinitionLModules}
Let $W\subseteq\Xhat$ be an admissible subset and define $\IrrRep(\L_W) = \coprod_{P\in\Pl(W)} \IrrRep(L_P)$.

An \emph{$\L$-module} $\M = (E_\cdot, f_{\cdot\cdot})$ on $W$ consists of a graded regular $L_P$-module $E_P$ for all $P\in\Pl(W)$ together with morphisms $f_{PQ}\colon H(\n_P^Q;E_Q)\to E_P[1]$ for all $P\le Q\in\Pl(W)$ such that
  \[
  \sum_{Q\in [P,R]} f_{PQ}[1]\circ H(\n_P^Q;f_{QR}) = 0
  \]
  for $P\le R\in\Pl(W)$.

  A morphism $\phi\colon \M \to \N$ between $\L$-modules $\M$ and $\N=(F_\cdot,g_{\cdot\cdot})$ consists of $L_P$-morphisms $\phi_{PQ}\colon H(\n_P^Q;E_Q) \to F_P$ for all $P\le Q\in\Pl(W)$ which satisfy
  \[
  \sum_{Q\in [P,R]} g_{PQ}\circ H(\n_P^Q;\phi_{QR}) =
    \sum_{Q\in [P,R]} \phi_{PQ}[1]\circ H(\n_P^Q;f_{QR})\ .
  \]
  
  The category of $\L$-modules $\Mod(\L_W)$ is an additive category.  The shift by $k$ of an $\L$-module $\M$ is defined by $\M[k]=(E_\cdot[k],(-1)^kf_{\cdot\cdot})$.  If $\N= (F_\cdot,g_{\cdot\cdot})$ is another $\L$-module the direct sum is $(E_\cdot\oplus F_\cdot,f_{\cdot\cdot}\oplus g_{\cdot\cdot})$.  The full subcategory of \emph{bounded $\L$-modules} $\Mod^b(\L_W)$ is defined by requiring $E_P$ to be an object of $\Graded^b(L_P)$ for all $P\in\Pl(W)$.

Note that if $W=X_P$ then $\Mod(\L)$ is the category of complexes of regular $L_P$-modules.  In general, though, $\Mod(\L)$ is not defined as a category of complexes.

\subsection{Functors}
\label{ssectFunctors}
Let $k\colon Z \hookrightarrow W$ be an inclusion of admissible subsets.  Define the functor $k^!\colon \Mod(\L_W)\to \Mod(\L_Z)$ by restricting the data of $\M\in \Mod(\L_W)$ to  $P,Q\in\Pl(Z)$.  When $Z$ is open in $W$ define $k^*=k^!$.  When $Z$ has a unique maximal stratum $k^*$ is defined in  \cite[\S3.4]{refnSaperLModules}; the only case we will use here is
\begin{equation*}
\i_P^*\M = \biggl( \bigoplus_{P\le R} H(\n_P^R;E_R), 
\sum_{P\le R\le S} H(\n_P^R;f_{RS}) \biggr)\ ,
\end{equation*}
the complex computing local cohomology at $P$.  Define the functor $k_*\colon  \Mod(\L_Z)\to \Mod(\L_W)$ by extending the data of an $\L$-module on $Z$ by $0$ if any index is not in $\Pl(Z)$.  We will only use $k_!$ here when $Z$ is closed in $W$ in which case $k_!=k_*$.

If $i$ (resp.~ $j$) is the inclusion of a closed (resp.~ an open) admissible subset of $W$ then  $i^*$ is a left adjoint to $i_*=i_!$ and $j^!$ is a right adjoint to $j_!$.  This will be used in \S\ref{sectMorphismsWithWeightedCohomology}.

For $P\in\Pl(W)$ the $\L$-module $i_P^!\M=(E_P,f_{PP})$ is the complex computing the local cohomology supported on $X_P$.  More generally, for $P\le Q\in\Pl(W)$, the $\L$-module $i_{P*}\ihat_Q^!\M$ is the complex computing the local cohomology along $X_P$ supported on $\Xhat_Q$.  This will be used to define micro-support in \S\ref{ssectMicroSupport}.

If $\M\in\Mod(\L_W)$ and $P\le Q\le Q'\in \P(W)$ there is a long exact sequence relating the local cohomology along $X_P$ supported on $\Xhat_Q$ and on $\Xhat_{Q'}$ \cite[(3.6.4)]{refnSaperLModules}:
\begin{equation}\label{eqnCompareLocalCohomology}
  \dots \to H^{i}(i_P^*\ihat_Q^!\M) \to  H^{i}(i_P^*\ihat_{Q^{\prime}}^!\M) \to H^{i}(i_P^*\jhat_{Q*}\jhat_Q^*\ihat_{Q^{\prime}}^!\M) \to \dots\ .
\end{equation}
Set $P'=(P,Q)\cap Q'$.  There are two spectral sequences abutting to the third term \cite[Lemma 3.7]{refnSaperLModules}: the \emph{Fary spectral sequence} with
\begin{equation}\label{eqnFarySpectralSequence}
  E_1 = \bigoplus_{P< \widetilde P\le P'} H(\n_P^{\widetilde P}; H(i_{\widetilde P}^* \ihat_{\widetilde P\vee Q}^!\M))\ ,
\end{equation}
and the \emph{Mayer-Vietoris spectral sequence} with
\begin{equation}\label{eqnMayerVietorisSpectralSequence}
  E_1 = \bigoplus_{P< \widetilde P\le P'} H(\n_P^{\widetilde P}; H(i_{\widetilde P}^* \ihat_{Q'}^!\M))[-\#\D_P^{\widetilde P}+1]\ .
\end{equation}

Finally suppose $W$ has a unique maximal stratum $X_S$.  If $\M\in\Mod(\L_W)$ and $R\le Q\in\P(W)$ there are natural morphisms
\begin{equation}\label{eqnNaturalMorphisms}
  i_R^!\M = i_R^*\ihat_R^!\M \overset{\LeftNaturalMorphism}\longrightarrow i_R^*\ihat^!_Q\M \overset{\RightNaturalMorphism}\longrightarrow i_R^*\ihat_S^!\M =   i_R^*\M
\end{equation}
which will play an important role in \S\ref{sectEliminatingMicroSupport}.

\subsection{Micro-support}\label{ssectMicroSupport}
Assume $W$ has a unique maximal stratum $X_S$.  If $V\in\IrrRep(L_W)$ is an $L_P$-module, define $P\le  Q_V^W \le Q_V^{\prime W}$ by
\begin{align*}
\D_P^{Q_V^W} &= \{\,\al\in\D_P^S\mid\langle \xi_V+\hsr,\al\spcheck\rangle<0\,\}\text{ and }\\
\D_P^{Q_V^{\prime W}} &= \{\,\al\in\D_P^S\mid\langle \xi_V+\hsr,\al\spcheck\rangle\le 0\,\}\ .
\end{align*}
(Sometimes it is convenient to write $Q_V^S$ instead of $Q_V^W$ and when $S=G$ we omit it.)  For $Q \in [Q_V^W, Q_V^{\prime W}]$ we define the \emph{$Q$-type} of $\M$ to be the cohomology
\begin{equation}
  \label{eqnType}
  \type_{Q,V}(\M) = H(i_P^*\ihat_Q^!\M)_V\ .
\end{equation}
When $Q=Q_V^{\prime W}$ or $Q_V^W$ we use the shorthand $\type_{\mu,V}(\M)$ and $\type_{\nu,V}(\M)$ respectively; these labels will be justified later in Corollary ~\ref{corEtamSofWetaC}.

The \emph{weak micro-support} of an $\L$-module $\M$ on $W$ is defined to be
\begin{equation*}
  \wmS(\M) = \{\, V\in \IrrRep(\L_W) \mid \type_{Q,V}(\M)\neq0
    \text{ for some } Q\in[Q_V^W,Q_V^{\prime W}]\,\}\ .
\end{equation*}
We similarly define $\muwmS(\M)$ and $\nuwmS(\M)$ by using $\type_{\mu,V}(\M)$ and $\type_{\nu,V}(\M)$ respectively.

Define the \emph{\textup(strong\textup) micro-support} $\mS(\M)$ as the subset of  $\wmS(\M)$ whose the elements $V\in\IrrRep(\L_W)\cap\Mod(L_P)$ satisfy the additional conjugate self-contragradient condition $(V|_{M_P})^* \cong \overline{V|_{M_P}}$.  Similarly define $\mS_\mu(\M)$ and $\mS_\nu(\M)$.

\subsection{Vanishing theorem}
By Theorem~4.1 of \cite{refnSaperLModules} there is a functor $\Sh_W$ from $\Mod(\L_W)$ to $\Der_\X(W)$, the derived category of constructible sheaves on $W$.  One incarnation of $\Sh_W(\M)$ is the direct sum over $P\in\Pl(W)$ of sheaves of special differential forms \cite[\S13]{refnGoreskyHarderMacPherson} on $\Xhat_P$ with coefficient system $\EE_P$; the differential arises from  exterior differentiation, restriction to boundary strata, and the $f_{PQ}$.  This functor commutes with the functors on $\L$-modules defined in \S\ref{ssectDefinitionLModules}.   The \emph{cohomology $H(W;\M)$ of an $\L$-module $\M\in\Mod(\L_W)$} is defined to be the hypercohomology $H(W;\Sh_W(\M))$.   We have the following vanishing theorem

\begin{thm}[{\cite[\S\S10.4,10.6]{refnSaperLModules}}]
  \label{thmVanishingTheorem}
  If $\mS(M)$, $\mS_\mu(\M)$, or $\mS_\nu(\M)$ are empty then  $H(W;\Sh_W(\M))=0$.%
  \footnote{Actually the theorem states the cohomology vanishes for degrees $i\notin [c(\M),d(\M)]$.  This interval is the smallest containing all sums $j+k$, where $\type_{Q,V}^j(\M)\neq0$ and $H^k_{(2)}(X_P;\VV)\neq0$ for any $V\in\mS(\M)\cap\IrrRep(L_P)$ and $Q\in[Q_V^W,Q_V^{\prime W}]$.  We will not need this more detailed information however.}
\end{thm}

\section{Homotopy category $\K(\L_W)$ of $\L$-modules}
\label{sectHomotopyCategory}
Let $W$ be an admissible set with a unique maximal stratum. Consider $\L$-modules $\M=(E_\cdot,f_{\cdot\cdot})$ and $\N=(F_\cdot,g_{\cdot\cdot})$ on $W$.  Two $\L$-morphisms $\phi_1$, $\phi_2\colon \M \to \N$ are homotopic \cite[\S3.9]{refnSaperLModules} if there are degree $-1$ morphisms
\[
h_{PQ}\colon H(\n_P^Q;E_Q) \longrightarrow F_P[-1] \qquad\text{for all $P\le Q\in\Pl(W)$}
\]
such that
\begin{multline}
  \phi_{1PR}-\phi_{2PR}= \sum_{P\le Q\le R} g_{PQ}[-1]\circ H(\n_P^Q; h_{QR})\\
  +\sum_{P\le Q\le R} h_{PQ}[1]\circ H(\n_P^Q;f_{QR}) \qquad \text{for all $P\le R\in\Pl(W)$.}
\end{multline}
Let $\K(\L_W)$ be the homotopy category of $\L$-modules on $W$; its morphisms are the homotopy classes of $\L$-morphisms.  Let $\K^b(\L_W)$ be the full subcategory whose objected are bounded $\L$-modules.

The \emph{mapping cone} of a morphism of $\L$-modules $\phi\colon \M \to \N$ is  the $\L$-module
\begin{equation}\label{eqnMappingCone}
  \cone(\phi)= \left( E_\cdot[1] \oplus F_\cdot, \begin{pmatrix} -f_{\cdot\cdot} & 0 \\ \phi_{\cdot\cdot} & g_{\cdot\cdot}\end{pmatrix}\right) \ .
\end{equation}
Define natural morphisms $\LeftMappingConeMorphism(\phi)\colon \N \to \cone(\phi)$ and $\RightMappingConeMorphism(\phi) \colon \cone(\phi) \to \M[1]$  by
\[
\LeftMappingConeMorphism(\phi)_{PP}= \begin{pmatrix}0\\ \id_{F_P}\end{pmatrix}\qquad \text{and}\qquad 
\RightMappingConeMorphism(\phi)_{PP}= \begin{pmatrix} \id_{E_P[1]} & 0\end{pmatrix}
\]
  for all $P\in\Pl(W)$.

The usual proof for the homotopy category of complexes, for example in \cite[\S\S1.4, 1.5]{refnKashiwaraSchapira}, generalizes to show that $\K(\L_W)$ is a triangulated category with the above definition of mapping cone; a distinguished triangle in $\K(\L_W)$ is a diagram isomorphic to
\begin{equation*}
  \cdots  \longrightarrow \M \xrightarrow{\ \phi\ } \N \xrightarrow{\LeftMappingConeMorphism(\phi)} \cone(\phi)
\xrightarrow{\RightMappingConeMorphism(\phi)} \M[1] \longrightarrow \cdots
\end{equation*}
for any morphism $\phi\colon \M \to \N$.  It is straightforward to check that $k^!$, $k^*$, $k_*$, and $k_!$ in the cases defined in \S\ref{ssectFunctors} are triangulated functors.

If $W=X_P$ one can show the functor $H(\cdot)\colon \Complex(L_P) \to \Graded(L_P)$ is cohomological by applying it to $\N\to \cone(\phi)\to\M[1]$ and then following the proof in  \cite[Prop.~1.5.6]{refnKashiwaraSchapira}.%
\footnote{For all admissible $W$ one can similarly show the global cohomology functor $H^0(W;\cdot)$ is cohomological.  We will not use this here.}
As a consequence for $V\in\IrrRep(W)$ and $Q\in[Q_V^W,Q_V^{\prime W}]$ we have a long exact sequence of $Q$-type \eqref{eqnType}:
\begin{equation}\label{eqnLESType}
\cdots  \rightarrow \type_{Q,V}^d(\M) \rightarrow \type_{Q,V}^d(\N) \rightarrow \type_{Q,V}^d(\cone(\phi)) 
          {\rightarrow} \type_{Q,V}^{d+1}(\M) \rightarrow \cdots\ .
\end{equation}

We note two distinguished triangles.
Assume $P$ is minimal within $\Pl(W)$ so that $X_P$ is closed in $W$ and hence $i_{P!}$ is defined.  
Then for $\M\in\K(\L_W)$ there are distinguished triangles:
\begin{gather}
\label{eqnFirstDistinguishedTriangle}
  \cdots  \longrightarrow i_{P!}i_P^! \M \longrightarrow \M \longrightarrow j_{P*}j_P^*\M 
           \longrightarrow  i_{P!}i_P^! \M[1] \longrightarrow \cdots \\
\label{eqnsecondDistinguishedTriangle}
  \cdots  \longrightarrow j_{P!}j_P^! \M \longrightarrow \M \longrightarrow i_{P*}i_P^*\M 
           \longrightarrow  j_{P!}j_P^! \M[1] \longrightarrow \cdots\makebox[0pt]{\ .}
\end{gather}
The proof of \eqref{eqnFirstDistinguishedTriangle} for example is to first note that $j_{P*}j_P^*\M$ has the data of $\M$ for indices $>P$ while  $ i_{P!}i_P^! \M[1]$ has the data of $\M[1]$ for $P$.  The morphism $j_{P*}j_P^*\M \to i_{P!}i_P^! \M[1]$ is given by $-f_{PQ}$ for $Q>P$ which results in its mapping cone as defined in \eqref{eqnMappingCone} being precisely $\M[1]$.

\begin{prop}\label{propQuasiIsomorphismIsHomotopyIsomorphism}
  If $\phi\colon \M\to \N$ induces an isomorphism on local cohomology for all $P\in\Pl(W)$ then $\phi$ is an isomorphism in the homotopy category.
\end{prop}
\begin{rem*}
Hence a quasi-isomorphism of $\L$-modules is already an isomophism in the homotopy category, unlike the situation for complexes of sheaves.  Thus there is no need to invert quasi-isomorphisms and pass to a derived category.
\end{rem*}

\begin{proof}
  The proof is by induction on $\#\Pl(W)$.  The case $\#\Pl(W)=1$, $\L$-modules on a single stratum $X_P$, is clear since $\Mod(L_P)$ is a semi-simple category.  In the general case, let $X_P\subset W$ be a minimal stratum.  We know $i_P^*\phi$ is an isomorphism and by induction we know $j_{P*}j_P^*\phi$ is a homotopy isomorphism.  Thus the distinguished triangle $\to i_P^! \M \to i_P^* \M \to i_P^*j_{P*}j_P^*\M \to$ (apply $i_P^*$ to \eqref{eqnFirstDistinguishedTriangle}) shows that $i_P^!\phi$ is an isomorphism.  Finally \eqref{eqnFirstDistinguishedTriangle} then shows $\phi$ is a homotopy isomorphism.
\end{proof}

\section{Partial orders on $\IrrRep(\L_W)$}
\label{sectPartialOrder}

Let $W\subseteq \Xhat$ be admissible with unique maximal stratum $X_S$ and $P,R\in\P(W)$.

\subsection{The partial order $\pmb{\preccurlyeq}$}
\label{ssectPartialOrder}
We recall the partial order $\preccurlyeq$ on $\IrrRep(\L_W)$ and its variants from \cite[\S\S9.1, 22.3]{refnSaperLModules}.  Suppose $V_P\in\IrrRep(L_P)$ and $V_{R}\in\IrrRep(L_{R})$.  Define $V_P\preccurlyeq V_{R}$ if
\begin{enumerate}[leftmargin=*]
\item \label{itemParabolic} $P\le R$ and
\item \label{itemKostantComponent} $V_P=H^{\l(w)}(\n_P^{R};V_{R})_w$ for some $w\in W_P^{R}$.
\end{enumerate}
The $w$ here is unique and we let $[V_{R}:V_P]$ denote $\l(w)$.

We define $V_P \preccurlyeq_+ V_{R}$ (resp.\ $V_P\preccurlyeq_{-} V_{R}$) if, in addition to \ref{itemParabolic} and \ref{itemKostantComponent},
\begin{enumerate}[leftmargin=*]
  \setcounter{enumi}{2}
  \item \label{itemRootCone} $(\xi_{V_P}+\hsr)|_{\sa_P^{R}} \in
    \lsp+\sa_P^{R*}$ \quad (resp.\ $(\xi_{V_P}+\hsr)|_{\sa_P^{ R}} \in
    -\lsp+\sa_P^{ R*}$) \ .
\end{enumerate}
We write $V_P\preccurlyeq_0 V_{R}$ if $V_P \preccurlyeq_+ V_{R}$ and $V_P\preccurlyeq_- V_{R}$ both hold.%
\footnote{Examples of $V_P\prec_0 V_{R}$ always occur when $X$ is not equal-rank \cite[Lem.~8.8]{refnSaperLModules}.}

Finally define $V_P \preccurlyeq_{+,\str} V_{R}$ (resp.\ $V_P\preccurlyeq_{-,\str} V_{R}$) if, in addition to \ref{itemParabolic} and \ref{itemKostantComponent},
\begin{enumerate}[leftmargin=*,label=(\roman*)$_{\textup{str}}$]
  \setcounter{enumi}{2}
  \item \label{itemIntRootCone} $(\xi_{V_P}+\hsr)|_{\sa_P^{R}} \in
    \Int\lsp+\sa_P^{R*}$ \quad (resp.\ $(\xi_{V_P}+\hsr)|_{\sa_P^{ R}} \in
    -\Int\lsp+\sa_P^{ R*}$)\ .
\end{enumerate}
In all these notations we replace $\preccurlyeq$ by $\prec$ if $P<R$.

\begin{lem}
  \label{lemQInequalityImpliesSignedOrder}
  Assume $V_P\preccurlyeq V_{R}$.  If $R\le Q_{V_P} ^{\prime W}$ \textup(resp.~  $R\le Q_{V_P} ^{W}$\textup) then $V_P\preccurlyeq_- V_{R}$ \textup(resp.~ $V_P\preccurlyeq_{-,\str} V_{R}$\textup).  If $R\le (P,Q_{V_P}^W)\cap S$ \textup(resp.~  $R\le (P,Q_{V_P}^{\prime W})\cap S$\textup) then $V_P\preccurlyeq_+ V_{R}$ \textup(resp.~ $V_P\preccurlyeq_{+,\str} V_{R}$\textup).
\end{lem}
\begin{proof}
The lemma follows from the fact that $\sa_P^{R+*}\subseteq \lsp+\sa_P^{R*}$ \cite[IV, \S6.2]{refnBorelWallach} and the similar inclusion of the interiors.
\end{proof}

\begin{lem}
  \label{lemInterleaveRepresentation}
If $V_P\prec V_{R}$ and $Q$ satisfies $P< Q <R$ then there exists a unique $V_{Q}\in\IrrRep(L_{Q})$ such that $V_P\prec V_{Q} \prec V_{R}$.  If  $V_P=H(\n_P^R;V_R)_w$ then $V_{Q}=H^{\l(w_{Q})}(\n_{Q}^{R};V_{R})_{w_{Q}}$ where $w=w^{Q}w_{Q}\in W_P^{Q}W_{Q}^{R}$.  The equality $[V_{R}:V_P]=[V_{R}:V_{Q}]+[V_{Q}:V_P]$ holds.
\end{lem}
\begin{proof}
  The lemma follows from Kostant's theorem; see \S\ref{ssectKostantsTheorem}.
\end{proof}

\begin{lem}
  \label{lemProjectPartialOrder}
  \leavevmode
  \begin{enumerate}[leftmargin=0em, labelsep=.5em, itemindent=2em]
  \item\label{itemCharacterProjection}
    If $V_P\preccurlyeq V_{Q}$ then $(\xi_{V_P}+\hsr_P)\vert_{\sa_{Q}}=\xi_{V_{Q}}+\hsr_{Q}$.
  \item
    \label{itemProjectionOfPartialOrder}
    If $V_P\preccurlyeq V_{Q}\preccurlyeq V_R$ for some $V_R \in\IrrRep(L_R)$ then
    \[
    V_R\succcurlyeq_+ V_P \quad\Longrightarrow\quad V_R\succcurlyeq_+  V_{Q} \quad\text{and}\quad   V_R\succcurlyeq_{+,\str} V_P \quad\Longrightarrow\quad V_R\succcurlyeq_{+,\str}  V_{Q}
    \]
    and similarly for $\preccurlyeq_-$ and $\preccurlyeq_{-,\str}$.
  \end{enumerate}
\end{lem}

\begin{proof}
For \ref{itemCharacterProjection}, Kostant's theorem shows that $V_P=H^{\l(w)}(\n_P^{Q};V_{Q})_{w}$ has highest weight $w(\lambda_{Q}+\hsr)-\hsr$ where $\lambda_{Q}$ is the highest weight of $V_{Q}$.  Thus $(\xi_{V_P}+\hsr_P)\vert_{\sa_{Q}}=(w(\lambda_{Q}+\hsr))\vert_{\sa_{Q}}=(\lambda_{Q}+\hsr)\vert_{\sa_{Q}} =\xi_{V_{Q}}+\hsr_{Q}$ since $w\in W_P^{Q}$ acts trivially on $\sa_{Q}$.

For the first implication of \ref{itemProjectionOfPartialOrder} recall the fact that the projection $\sa_P^{R*}\to\sa_{Q}^{R*}$ preserves the root cone and its interior \cite[Lemma 3.2]{refnSaperTilings}.  Thus $(\xi_{V_P} + \hsr)|_{\sa_P^R} \in \lsp+ \sa_P^{R*}$ implies  $(\xi_{V_P} + \hsr)|_{\sa_{Q}^R} \in \lsp+ \sa_{Q}^{R*}$.  Now apply \ref{itemCharacterProjection}.  The other implications are similar.
\end{proof}

\subsection{The partial order $\pmb{\leqslant_{\eta}}$}
\label{ssectNewPartialOrder}

We now combine variants of $\prec_-$ and $\succ_{+}$ into two new partial orders on $\IrrRep(\L_W)$.  First a

\begin{lem}
  \label{lemNewPartialOrder}
If either $V_{P}\succcurlyeq_+ V_R$ or $V_P\preccurlyeq_- V_{R}$ then $\xi_{V_{P}}+\hsr_{P} \le \xi_{V_R}+\hsr_R$ \textup(after extension by $0$ to $\sa_{P_0}^{G*}$\textup). 
For $\succ_{+,\str}$ or $\prec_{-,\str}$ we obtain $\lneq$ on the characters.
\end{lem}

\begin{proof}
If $V_{P}\succcurlyeq_+ V_R$ then, by Lemma ~\ref{lemProjectPartialOrder}\ref{itemCharacterProjection}, $(\xi_{V_R}+\hsr_R)\vert_{\sa_{P}}=\xi_{V_{P}}+\hsr_{P}$.  
Since $\sa_R=\sa_{P} + \sa_R^{P}$, $(\xi_{V_R}+\hsr_R) - (\xi_{V_{P}}+\hsr_{P}) = (\xi_{V_R}+\hsr_R)\vert_{\sa_R^{P}}\in \lsp+\sa_R^{P*} \subseteq \lsp+\sa_{P_0}^{G*}$  (the inclusion holds by \cite[Remark~ 3.3(ii)]{refnSaperTilings}).  If $V_{P}\succ_{+,\str} V_R$ we must show in addition that $\xi_{V_{P}}+\hsr_{P} \neq \xi_{V_R}+\hsr_R$; this follows since the difference belongs to $ \Int\lsp+\sa_R^{P*}$ and $P\neq R$.  The proofs for $\preccurlyeq_-$ and  $\prec_{-,\str}$ are similar.
\end{proof}

Let $\LprecLeft$ be the smallest transitive relation on $\IrrRep(\L_W)$ for which $V_P\LprecLeft V_{R}$ holds in any of these three cases:
\begin{enumerate}[leftmargin=*,label=(\alph*),widest=xxxx]
\item\label{itemEqual} $V_P=V_{R}$,
\item\label{itemPrec} $V_P \prec_- V_{R}$, and
\item\label{itemSucc} $V_P \succ_{+,\str} V_{R}$.
\end{enumerate}
Likewise define $\LprecRight$ by replacing $\prec_-$ and $\succ_{+,\str}$ with $\prec_{-,\str}$ and $\succ_+$ respectively.

\begin{lem}
  The relations $\LprecLeft$ and $\LprecRight$ on $\IrrRep(\L_W)$ are partial orders.
\end{lem}

\begin{proof}
  We will check antisymmetry for $\LprecLeft$.  Suppose $V_P\LprecLeft V_{R}$ and $V_{R}\LprecLeft V_P$.  Then there is a sequence $V_P=V_0,V_1,\dots,V_N=V_{R},V_{N+1},\dots, V_M=V_P$ where $V_i\in \IrrRep(P_i)$ and $V_{i-1}\LprecLeft V_{i}$ for $1\le i \le M$ by one of cases \ref{itemEqual}--\ref{itemSucc} above.  By Lemma ~\ref{lemNewPartialOrder}, $\xi_{V_{i-1}}+\hsr_{P_{i-1}} \le \xi_{V_{i}}+\hsr_{P_{i}}$ for $1\le i \le M$.  However $V_0=V_M$ so we must have equality at every step.  This implies, again by Lemma ~\ref{lemNewPartialOrder}, that no generating relation of the strict case \ref{itemSucc} can occur in this sequence.  Thus $V_P=V_{P_0}\preccurlyeq_- \dots\preccurlyeq_- V_N=V_{R} \preccurlyeq_-\cdots\preccurlyeq_- V_{P_M}=V_P$ so all $V_i$ are equal and in particular $V_P=V_{R}$.
\end{proof}
    
\begin{rem*}
Imposing strictness on one side is essential to obtaining a partial order since otherwise $V_P \LprecLeft V_{R}$ and $V_{R}\LprecLeft V_P$ would both hold when $V_P\neq V_{R}$ satisfy $V_P\prec_0 V_{R}$.
\end{rem*}

\begin{rem}
  In view of Lemma ~\ref{lemNewPartialOrder}, the partial orders $\LprecLeft$ and $\LprecRight$ on  $\IrrRep(\L_W)$ are likely related to the filtration on automorphic forms constructed by Franke in \cite{refnFranke}.
\end{rem}

\section{Weighted cohomology}
\label{sectWeightedCohomology}
In this section we recall from \cite[\S6]{refnSaperLModules} how to define Goresky, Harder, and MacPherson's weighted cohomology sheaf  \cite{refnGoreskyHarderMacPherson} as an $\L$-module and from \cite[\S16]{refnSaperLModules} how to calculate its micro-support.

We work with the weighted cohomology $\L$-module $\W^\eta\C(E_R)$ on $\Xhat_R$ where $R\in\Pl$, $E_R\in\Mod(L_R)$, and $\eta\in\sa^{G*}$.  We call $\eta\in\sa^{G*}$ a \emph{weight profile}; one can associate to $\eta$ a corresponding ``classical'' weight profile $p$ in the sense of \cite{refnGoreskyHarderMacPherson}.%
\footnote{%
Specifically for all $\al\in\D^R$ let $R_\al<R$ satisfy $\D^{R_\al}=\D^R\setminus\{\al\}$ and express $\eta_{R_\al}^R$ as $p_\al \chi_\al$ where $\chi_\al$ is the positive generator of $\Char(S_{R_\al}^R)$ with respect to $\al_{R_\al}$.  Then the weight profile $p\colon \D^R \to \ZZ +\frac12$ associated to $\eta$ is $p(\al) = \lfloor p_\al \rfloor - \frac12$.}

\subsection{Weight truncation of $\L$-modules}

If $Q\le R$, $E_Q\in \Mod(L_Q)$, and $\xi\in \Char(S_Q^R)\subset \sa_Q^{R*}$, let $E_{Q,\xi}\subseteq E_Q$ be the subspace on which $S_Q^R$ acts via $\xi$.  Thus $E_Q = \bigoplus_{\xi} E_{Q,\xi}$ and we set $\t^{\ge\eta^R} E_Q = \oplus_{\xi\ge\eta_Q^R} E_{Q,\xi}$ and similarly $\t^{\not\ge\eta^R} E_Q$ and $\t^{<\eta^R} E_Q$.
There is a canonically split short exact sequence $0\to \t^{\ge\eta^R} E_Q \to E_Q \to \t^{\not\ge\eta^R} E_Q\to 0$.

Given  an $\L$-module $\M$ on $\Xhat_R$, its \emph{$\eta$-weight truncation} along the $X_Q$ stratum is the mapping cone
\[
\t_Q^{\ge\eta^R} \M = \cone(\M \to i_{Q*}\t^{\not\ge \eta^R}i_Q^*\M)[-1]\ .
\]

\subsection{Weighted cohomology as an $\L$-module}
The \emph{weighted cohomology $\L$-module $\W^\eta\C(E_R)$} for $E_R\in\Mod(L_R)$ is 
\[
\W^\eta\C(E_R) = \t_{Q_1}^{\ge\eta^R}j_{Q_1*}\cdots\t_{Q_N}^{\ge\eta^R}j_{Q_N*}i_{R*}E_R
\]
where $Q_1,\dots,Q_N$ is an enumeration of $\Pl(\Xhat_R)\setminus\{R\}$  such that if $Q_i<Q_j$ then $i<j$.  This agrees with the definition in \cite[\S6]{refnSaperLModules} and is independent of the choice of ordering.
The realization of $\W^\eta\C(E_R)$ in the derived category is the weighted cohomology sheaf of Goresky, Harder, and MacPherson \cite{refnGoreskyHarderMacPherson} for the associated ``classical'' weight profile  $p$.

\begin{prop}
  Let $E_R\in\Mod(L_R)$ and $P\le R$.  Then
  \begin{align}
    \label{eqnLocalWCohomology}
    H(i_P^*\W^\eta\C(E_R)) &= \tau^{\ge\eta^R}H(\n_P^R;E_R)\ ,\\
    \label{eqnLocalWCCohomologyWithSupports}
    H(i_P^!\W^\eta\C(E_R)) &= \tau^{<\eta^R}H(\n_P^R;E_R)[-\#\D^R_P]
  \end{align}
  and there is a split short exact sequence for the link cohomology
  \begin{equation}
0\to H(i_P^*\W^\eta\C(E_R)) \to    H(i_P^*j_{P*}j_P^*\W^\eta\C(E_R)) \to  H(i_P^!\W^\eta\C(E_R))[1] \to 0\ .
  \end{equation}
\end{prop}

\begin{proof}
  See \cite[\S\S16.1,16.2]{refnSaperLModules} and \cite[(18.2)]{refnGoreskyHarderMacPherson}.
\end{proof}

\subsection{Middle weight profiles}
For $P\le R$ we will be using the orders on $\sa_P^{R*}$ from \eqref{eqnNewGE} and \eqref{eqnNewLT}.

\begin{lem}
  \label{lemEpsisonCondition}
  There exists $\epsilon>0$ such that for all $P\le R \in \Pl$ and all characters $\psi\in\Char(S_P^R)$ we have
  \begin{alignat}{3}
    \label{eqnPositiveImpliesGreaterEpsilon}
    \psi&>0 &\quad &\Longrightarrow\quad & \psi &> \epsilon\hsr_P^R \qquad\text{and}\\
    \label{eqnLessEpsilonImpliesNegative}
    \psi &< \epsilon\hsr_P^R &\quad&\Longrightarrow \quad &\psi &\le 0\ .
  \end{alignat}
\end{lem}

\begin{proof}
The $\ZZ$-span of $\D_P^R$ has finite index in the character lattice $\Char(S_P^R)$; let $N_{P,R}$ be its index.  Since $\hsr_P^R\in \Int \lsp{+}\sa_P^{R*}$ we know $\langle \hsr_P^R,\b\spcheck\rangle>0$ for all $\b\in\widehat\D_P^R$.  Choose $\epsilon < N_{P,R}^{-1} \langle \hsr_P^R,\b\spcheck\rangle^{-1}$ for all $P\le R$ and all $\b\in\widehat\D_P^R$.  Consider $\psi\in\Char(S_P^R)$.  Then $N_{P,R}\psi = \sum_{\al\in\D_P^R} c_\al \al$ where all $c_\al\in\ZZ$.  If $\psi>0$ and  $\beta\in\Dhat_P^R$ corresponds to some $\al\in\D_P^R$ then $\langle N_{P,R}\psi,\b\spcheck\rangle=c_\al\ge 1$.  Thus $\langle \psi,\b\spcheck\rangle \ge N_{P,R}^{-1}>\epsilon  \langle \hsr_P^R,\b\spcheck\rangle$ which proves \eqref{eqnPositiveImpliesGreaterEpsilon}.  The proof of the contrapositive of  \eqref{eqnLessEpsilonImpliesNegative} is similar.
\end{proof}

The \emph{upper and lower middle weight profiles} $\mu$ and $\nu\in\sa^{G*}$ are defined by
\[
\mu = -\hsr + \epsilon\hsr \quad\text{and}\quad \nu=-\hsr
\]
where $\epsilon>0$ is as in Lemma ~\ref{lemEpsisonCondition}. For $\xi\in \Char(S_P^R)$ we have the following middle weight profile truncations:
\begin{alignat}{2}
  \label{eqnGreaterThanMu}
  \xi\ge \mu_P^R &\quad\Longleftrightarrow \quad \xi +\hsr_P^R\in \Int\lsp{+}\sa_P^{R*} \quad&&\Longleftrightarrow\quad V_P\prec_{+,\str} V_{R}\ ,\\
  \label{eqnGreaterThanNu}
  \xi\ge \nu_P^R &\quad\Longleftrightarrow \quad \xi +\hsr_P^R\in \lsp{+}\sa_P^{R*} \quad&&\Longleftrightarrow\quad V_P\prec_{+} V_{R}\ ,\\
  \label{eqnStrictlyLessThanMu}
  \xi< \mu_P^R &\quad\Longleftrightarrow \quad \xi +\hsr_P^R\in -\lsp{+}\sa_P^{R*}\quad &&\Longleftrightarrow\quad V_P\prec_{-} V_{R} \ ,\\
  \label{eqnStrictlyLessThanNu}
  \xi< \nu_P^R &\quad\Longleftrightarrow \quad \xi +\hsr_P^R\in -\Int\lsp{+}\sa_P^{R*}\quad &&\Longleftrightarrow\quad V_P\prec_{-,\str} V_{R}\ .
\end{alignat}
The first equivalence in each line follows from the definitions above; in the cases involving $\mu$ we apply Lemma ~\ref{lemEpsisonCondition} to $\psi=\xi+\hsr_P^R$.  The second equivalence in each line are the definitions of the partial orders from \S\ref{ssectPartialOrder}.

In particular this shows that our $\W^\mu\C(E_R)$ and $\W^\nu\C(E_R)$ correspond to the ``classical'' upper and lower weight profiles from \cite{refnGoreskyHarderMacPherson}.

\subsection{The micro-support of weighted cohomology}
For a middle weight profile the micro-support of the weighted cohomology $\L$-module was calculated in \cite[Theorem~ 16.3]{refnSaperLModules}:

\begin{prop}
  \label{propMicroSupportWeightedCohomology}
  For $P\le R$ let $V_R \in \IrrRep(L_R)$ and $V_P \in \IrrRep(L_P)$.  Let $\eta$ be a middle weight profile.  Then $V_P\in \wmS(\W^\eta\C(V_R))$ if and only if $V_P \preccurlyeq_0 V_R$.  Furthermore\textup:
  \begin{enumerate}[leftmargin=0em, labelsep=.5em, itemindent=2em]
  \item For such $V_P$,
    \begin{equation}\label{eqnwmSForWmuC}
    H(i_P^*\ihat_Q^!(\W^\mu\C(V_R))_{V_P} = V_P[-[V_R:V_P]-\#\D_P^R]
    \end{equation}
    when $Q=Q_{V_P}^R=P$ and is zero otherwise.
  \item For such $V_P$,
    \begin{equation}\label{eqnwmSForWnuC}
    H(i_P^*\ihat_Q^!(\W^\nu\C(V_R))_{V_P} = V_P[-[V_R:V_P]]
    \end{equation}
    when $Q=Q_{V_P}^{'R}=R$ and is zero otherwise.
  \end{enumerate}
\end{prop}

\begin{cor}
  \label{corEtamSofWetaC}
  For $\eta$ a middle weight profile,
  \[
  \etawmS(\W^\eta\C(V_R))= \{V_R\} \quad\text{and}\quad \type_{\eta,V_R}(\W^\eta\C(V_R))=V_R\ .
  \]
\end{cor}
\begin{proof}
If $V_P \preccurlyeq_0 V_R$ then $Q_{V_P}^R=P$ and  $Q_{V_P}^{'R}=R$.  By definition (see \S\ref{ssectMicroSupport}) $V_P$ contributes to $\muwmS$ when \eqref{eqnwmSForWmuC} is nonzero for $Q= Q_{V_P}^{'R}$.  By the proposition this means $Q_{V_P}^R=Q_{V_P}^{'R}$ and therefore $P=R$.  We similarly treat $\nuwmS$.
\end{proof}

Let $W^\eta H(\Xhat_R;\VV_R) = H(\Xhat_R;\WnC(V_R))$ be the global weighted cohomology for weight profile $\eta$.

\begin{cor}
  \label{corWeightedCohomologyVanishing}
For $\eta$ a middle weight profile, $W^\eta H(\Xhat_R;\VV_R) = 0$ if $(V_R|_{M_R})^* \not\cong \overline{V_R|_{M_R}}$.
\end{cor}

\begin{proof}
Since $V_R$ fails the conjugate self-contragradient condition, Corollary ~\ref{corEtamSofWetaC} implies $\mS_\eta(\W^\eta\C(V_R)) =\emptyset$.  Now apply the vanishing Theorem ~\ref{thmVanishingTheorem}.
\end{proof}

\section{Morphisms to and from weighted cohomology}
\label{sectMorphismsWithWeightedCohomology}

We consider $W$ an admissible subset with a unique maximal stratum $X_S$.  The following proposition gives conditions on when a morphism to or from $\W^\eta\C(V_R)$ can be built up in the homotopy category starting with a morphism on $X_R$.

\begin{prop}\label{propMappingWC}
  Let $\M\in\K(\L_{W})$ and fix $\eta\in\sa^{G*}$ and $d\in\ZZ$.  Consider $R\in \Pl(W)$ and $V_R\in\IrrRep(L_R)$.
\begin{enumerate}[leftmargin=0em, labelsep=.5em, itemindent=2em]
\item\label{itemLeft} Let $\phi_R\colon V_R[-d] \to i_R^!\M$ be a morphism.  Assume that
\begin{equation}\label{eqnLocalCoVanishing}
\text{$H^{[V_R:V_P]+d+1}(i_P^!\M)_{V_P}=0$ for all $V_P\prec V_R$ such that $\xi_{V_P}^R\ge \eta^R_P$.}
\end{equation}
  Then there exists a morphism
\[
  \phi\colon \ihat_{R!}\W^\eta\C(E_R)[-d]\to \M
\]
extending $\phi_R$.  The extension is unique if $H^{[V_R:V_P]+d}(i_P^!\M)_{V_P}=0$ for all $V_P$ as in \eqref{eqnLocalCoVanishing}.
\item\label{itemRight} Let $\psi_R\colon i_R^*\M \to E_R[-d]$ be a morphism.  Assume that
\begin{equation}\label{eqnLocalVanishing}
\text{$H^{[V_R:V_P]+\#\D_P^R+d-1}(i_P^*\M)_{V_P}=0$ for all $V_P\prec V_R$ such that $\xi_{V_P}^R < \eta^R_P$.}
\end{equation}
  Then there exists a morphism
  \[
  \psi\colon \M\to \ihat_{R*}\W^\eta\C(E_R)[-d]
  \]
  extending $\psi_R$.  The extension is unique if $H^{[V_R:V_P]+\#\D_P^R+d}(i_P^*\M)_{V_P}=0$ for all $V_P$ as in \eqref{eqnLocalVanishing}.
\end{enumerate}
\end{prop}

\begin{rem}\label{remMappingWCMiddleProfileConditions}
In this paper we only need the case where $\eta$ is a middle weight profile $\mu$ or $\nu$.  By \eqref{eqnGreaterThanMu} and \eqref{eqnStrictlyLessThanNu}, when $\eta=\mu$ the condition on $V_R$ in \eqref{eqnLocalCoVanishing} is $V_R \succ_{+,\str} V_P$ and when $\eta=\nu$ the condition on $V_R$ in \eqref{eqnLocalVanishing} is $V_P \prec_{-,\str} V_R$.
\end{rem}

\begin{proof}
  We prove \ref{itemLeft}; \ref{itemRight} is similar.  Consider any open admissible set $U \subseteq W$ which contains $X_R$ and let $k_U\colon U\hookrightarrow W$ be the inclusion.   We will prove that $\phi_R$ extends to $\phi_{U}\colon k_U^!\ihat_{R!}\W^\eta\C(E_R)[-d] \to k_U^!\M$ for any such $U$ by induction on $\#\Pl(U\cap\Xhat_R)$.  The case $\#\Pl(U\cap\Xhat_R)=1$ is simply the existence of $\phi_R$.  In general let $X_P$ be a minimal stratum in $U\cap\Xhat_R$.  Thus $X_P$ is closed in $U$ and we let $i_P\colon X_P\hookrightarrow U$ and $j_P\colon U\setminus X_P\hookrightarrow U$ be the inclusions.   By induction we can assume that $\phi_R$ extends to $\phi_{U\setminus X_P}\colon j_P^*k_U^!\ihat_{R!}\W^\eta\C(E_R)[-d] \to j_P^*k_U^!\M$.  Consider the diagram where the two rows are distinguished triangles from \eqref{eqnFirstDistinguishedTriangle}:
  \begin{equation*}
    \xymatrix@-5mm{
      {k_U^!\ihat_{R!}\W^\eta\C(E_R)[-d]} \ar[r]^-{u} \ar@{.>}[d]_{\phi_U} &
      {j_{P*}j_P^{*}k_U^!\ihat_{R!}\W^\eta\C(E_R)[-d]} \ar[r] \ar[d]_{j_{P*}(\phi_{U\setminus X_P})} &
      {i_{P!}i_P^!k_U^!\ihat_{R!}\W^\eta\C(E_R)[-d+1]} \ar@{.>}[d] \\
      {k_U^!\M} \ar[r] &
      {j_{P*}j_P^{*}k^!_U\M} \ar[r]^{v'} &
      {i_{P!}i_P^!k^!_U\M[1]}
    }
  \end{equation*}
  By \cite[Prop.~1.1.9]{refnBeilinsonBernsteinDeligne} the extension $\phi_U$ exists if and only if $v'j_{P*}(\phi_{U\setminus X_P})u=0$ in
  \begin{multline*}  \Mor_{\K(\L_U)}(k_U^!\ihat_{R!}\W^\eta\C(E_R)[-d],i_{P!}i_P^!k_U^!\M[1]) = \\ \Mor_{\Graded(L_P)}(H(i_P^*\W^\eta\C(E_R)[-d-1]),H(i_P^!\M))\ .
  \end{multline*}
  and is unique if
    \begin{multline*}  \Mor_{\K(\L_U)}(k_U^!\ihat_{R!}\W^\eta\C(E_R)[-d],i_{P!}i_P^!k_U^!\M) = \\ \Mor_{\Graded(L_P)}(H(i_P^*\W^\eta\C(E_R)[-d]),H(i_P^!\M)) = 0\ .
  \end{multline*}
However by \eqref{eqnLocalWCohomology} the irreducible constituents of $H(i_P^*\W^\eta\C(E_R))$ shifted by $d+1$ and $d$ are precisely those we assume in \eqref{eqnLocalCoVanishing} to vanish in $H(i_P^!\M)$.
\end{proof}

\section{Preparatory lemmas}
\label{sectPreparatoryLemmas}
Let $W\subseteq \Xhat$ be an admissible subset with a unique maximal stratum $X_S$.  Consider $\M\in\Mod(\L_W)$ and a middle weight profile $\eta$.  Corollary~ \ref{corEtamSofWetaC} suggests that the weighted cohomology ``building blocks'' needed to represent $\M$ as an iterated mapping cone are parametrized by $V_R\in\etawmS(\M)$.  Proposition ~\ref{propMappingWC} studied morphisms between $\M$ and $i_{R*}\WnC(V_R)$.  In the next section will be Theorem~\ref{thmRemoveExtremalMicroSupport}, our main technical result, which proves the existence of such morphisms.  We present here 
several lemmas needed for this theorem.

\begin{lem}\label{lemMicrosupportOfPullback}
  Let $\M\in\Mod(\L_W)$, $V_R\in\IrrRep(L_R)$ for $R\in\Pl(W)$, and $d\in\mathbb Z$.
  \begin{enumerate}[leftmargin=0em, labelsep=.5em, itemindent=2em]
\item\label{itemMicrosupportOfProperPullback}
  If $V_P\in \wmS(\ihat_R^!\M)$ with $H^d(i_P^*\ihat_Q^!\ihat_R^!\M)_{V_P}\neq0$ for some $Q\in[Q_{V_P}^R,Q_{V_P}^{\prime R}]$ then there exists $V_{\widetilde P}\in \muwmS(\M)$ such that  $V_{P} \preccurlyeq_{-} V_{\widetilde P}$ and $\type_{\mu,V_{\widetilde P}}^{\tilde d}(\M)\neq0$ for some $\tilde d\le -[V_{\widetilde P}:V_{P}]+d$.
\item\label{itemMicrosupportOfPullback}
  If $V_P\in \wmS(\ihat_R^*\M)$ with $H^d(i_P^*\ihat_Q^!\ihat_R^*\M)_{V_P}\neq 0$ for some $Q\in[Q_{V_P}^R,Q_{V_P}^{\prime R}]$ then there exists $V_{\widetilde P}\in \nuwmS(\M)$ such that  $V_{\widetilde P} \succcurlyeq_{+} V_{P}$ and $\type_{\nu,V_{\widetilde P}}^{\tilde d}(\M)\neq0$ for some $\tilde d\le -[V_{\widetilde P}:V_{P}]+d$.
\end{enumerate}
\end{lem}

\begin{proof}
  For \ref{itemMicrosupportOfProperPullback} (resp.~ \ref{itemMicrosupportOfPullback})
  \cite[Prop.~22.6]{refnSaperLModules} implies there exists $V_{P_1} \in\wmS(\M)$ with $V_{P}\preccurlyeq_- V_{P_1}$ (resp.~  $V_{P_1}\succcurlyeq_+ V_{P}$) and $H^{d_1}(i_{P_1}^*\ihat_Q^!\M)_{V_{P_1}}\neq0$ for some $Q\in [Q_{V_{P_1}}^W,Q_{V_{P_1}}^{\prime W}]$ and $d_1\le d - [V_{P_1}:V_{P}]$.  Then by \cite[Prop. ~9.2]{refnSaperLModules} there exists $V_{\widetilde P}\in \muwmS(\M)\cap \nuwmS(\M)$ with $V_{P_1}\preccurlyeq_0 V_{\widetilde P}$ and $\Im (H^{\tilde d}(i_{\widetilde P}^*\ihat_{Q_{V_{\widetilde P}}^{W}}^!\M)_{V_{\widetilde P}} \to H^{\tilde d}(i_{\widetilde P}^*\ihat_{Q_{V_{\widetilde P}}^{\prime W}}^!\M)_{V_{\widetilde P}})\neq 0$ for some $\tilde d\le - [V_{\widetilde P}:V_{P_1}] + d_1 \le -[V_{\widetilde P}:V_{P}]+d$.
\end{proof}

\begin{lem}\label{lemBased}
  Let $\M\in\Mod(\L_W)$ and let $\eta$ be a middle weight profile.  Say $V_R\in \etawmS(\M)$ and $\type_{\eta,V_R}^d(\M)\neq0$ for some $d$.  For $Q\in [R,S]$ consider the two natural morphisms from \eqref{eqnNaturalMorphisms}
  \[
  i_R^!\M \overset{\LeftNaturalMorphism}\longrightarrow i_R^*\ihat^!_{Q}\M \overset{\RightNaturalMorphism}\longrightarrow i_R^*\M \ .
\]
 \begin{enumerate}[leftmargin=0em, labelsep=.5em, itemindent=2em]
 \item\label{itemBasedLeftFails} If $\eta=\mu$, $Q=Q_{V_R}^{\prime W}$, and the map
   \begin{equation*}
     H^d(i_R^!\M)_{V_R}\xrightarrow{H(\LeftNaturalMorphism)} \type_{\mu,V_R}^d(\M)
   \end{equation*}
vanishes then there exists $V_{\widetilde R}\in\muwmS(\M)$ such that $V_R\prec_- V_{\widetilde R}$ and $\type_{\mu,V_{\widetilde R}}^{\tilde d}(\M)\neq0$ for some $\tilde d\le -[V_{\widetilde R}:V_R] + d$.
 \item\label{itemBasedRightFails}  If $\eta=\nu$, $Q=Q_{V_R}^{W}$, and the map
   \begin{equation*}
     \type_{\nu,V_R}^d(\M) \xrightarrow{H(\RightNaturalMorphism)} H^d(i_R^*\M)_{V_R}
   \end{equation*}
   vanishes then there exists $V_{\widetilde R}\in\nuwmS(\M)$ such that $V_{\widetilde R}\succ_+ V_{R}$ and $\type_{\nu,V_{\widetilde R}}^{\tilde d}(\M)\neq0$ for some $\tilde d< -[V_{\widetilde R}:V_R] + d$.
 \end{enumerate}
\end{lem}



\begin{proof}
  We prove \ref{itemBasedLeftFails}.  The long exact sequence \eqref{eqnCompareLocalCohomology} with $P\le Q\le  Q'$ replaced by $R\le R\le Q^{\prime W}_{V_R}$ becomes
  \[
  \dots\to
  H^{d}(i_R^!\M)_{V_R} \xrightarrow{H(\LeftNaturalMorphism)} \type_{\mu,V_R}^d(\M) \to H^{d}(i_R^*\jhat_{R*}\jhat_R^*\ihat_{Q^{\prime W}_{V_R}}^!\M)_{V_R} \to \cdots\ .
  \]
Since $H(\LeftNaturalMorphism)=0$ the third term is nonzero.  The Fary spectral sequence \eqref{eqnFarySpectralSequence} for this term shows there exists $R<R_1\le Q^{\prime W}_{V_R}$ such that $ H^d(\n_R^{R_1};H(i_{R_1}^!\M))_{V_R}\neq0$.
  Thus $H^{d-[V_{R_1}:V_R]}(i_{R_1}^!\M)_{V_{R_1}}\neq0$ for some $V_{R_1}$ which by Lemma~\ref{lemQInequalityImpliesSignedOrder} satisfies $V_R\prec_- V_{R_1}$.
  This nonvanishing implies $V_{R_1}\in \wmS(\ihat_{R_1}^!\M)$ so by Lemma~\ref{lemMicrosupportOfPullback}\ref{itemMicrosupportOfProperPullback} there exists  $V_{\widetilde R}\in \muwmS(\M)$ with $V_{R_1}\preccurlyeq_- V_{\widetilde R}$ and $\type_{\mu,V_{\widetilde R}}^{\tilde d}(\M)\neq0$ for some $\tilde d\le - [V_{\widetilde R}:V_{R_1}] + (d -[V_{R_1}:V_R]) = -[V_{\widetilde R}:V_R]+d$.  Finally $V_R\prec_- V_{\widetilde R}$ as desired.
  
  We now prove \ref{itemBasedRightFails}.  The long exact sequence \eqref{eqnCompareLocalCohomology} with  $P\le Q\le  Q'$ replaced by $R\le Q^{W}_{V_R} \le S$ becomes
  \[
  \dots\to
  H^{d-1}(i_R^*\jhat_{Q^{W}_{V_R}*}\jhat_{Q^{W}_{V_R}}^*\ihat_{S}^!\M)_{V_R} \to  \type_{\nu,V_R}^d(\M) \xrightarrow{H^d(\RightNaturalMorphism)} H^{d}(i_R^*\M)_{V_R} \to  \cdots\ .
  \]
  Since $H^d(\RightNaturalMorphism)=0$ the first term is nonzero.  The Mayer-Vietoris spectral sequence \eqref{eqnMayerVietorisSpectralSequence} for this term shows there exists $R< R_1 \le (R,Q^{W}_{V_R})$ such that
  \[
  H^{d-\#\D_R^{R_1}}(\n_R^{R_1};H(i_{R_1}^* \M))_{V_R}\neq0\ .
  \]
  Thus $H^{d-[V_{R_1}:V_R]-\#\D_R^{R_1}}(i_{R_1}^* \M)_{V_{R_1}}\neq0$ for some $V_{R_1}\succ V_R$ and in fact $V_{R_1}\succ_+ V_{R}$ by Lemma~\ref{lemQInequalityImpliesSignedOrder}.
  This nonvanishing implies $V_{R_1}\in \wmS(\ihat_{R_1}^*\M)$ so the proof concludes similarly to part \ref{itemBasedLeftFails} by using Lemma~\ref{lemMicrosupportOfPullback}\ref{itemMicrosupportOfPullback}.
\end{proof}

\begin{lem}\label{lemLocalNonvanishing}
  Let $\M\in\Mod(\L_{W})$, $V_R\in\IrrRep(L_R)$ for $R\in\Pl(W)$, and $k\in\mathbb Z$.
\begin{enumerate}[leftmargin=0em, labelsep=.5em, itemindent=2em]
\item\label{itemLeftFails}
  If there exists $V_P\prec V_R $ such that
  \[
  H^{[V_R:V_P]+k}(i_P^!\M)_{V_P} \neq0 \text{ and } V_R \succ_{+,\str} V_P
  \]
  then there exists $V_{\widetilde P}\in \muwmS(\ihat_R^!\M)$ such that  $V_{R} \succ_{+,\str} V_{\widetilde P}$ and $\type_{\mu,V_{\widetilde P}}(\ihat_R^!\M)$ is nonzero in some degree $\le [V_R:V_{\widetilde P}]+k$.
\item\label{itemRightFails}
  If there exists $V_P\prec V_R$ such that
  \[
  H^{[V_R:V_P]+\#\D_P^R+k}(i_P^*\M)_{V_P} \neq0 \text{ and } V_P \prec_{-,\str} V_R
  \]
  then there exists $V_{\widetilde P}\in \nuwmS(\ihat_R^*\M)$ such that $V_{\widetilde P} \prec_{-,\str} V_R$ and $\type_{\nu,V_{\widetilde P}}(\ihat_R^*\M)$ is nonzero in some degree $\ge [V_R:V_{\widetilde P}]+\#\D_{\widetilde P}^R+k$.
\end{enumerate}
\end{lem}

\begin{proof}
  The hypothesis $V_R \succ_{+,\str} V_P$ in \ref{itemLeftFails} implies that  $(\xi_{V_P}+\hsr)\vert_{\sa_P^R} \in \Int \sa_P^{*R+}$ (see \ref{itemIntRootCone} in \S\ref{ssectPartialOrder}).  Thus $Q^{\prime R}_{V_P}<R$ since if $Q^{\prime R}_{V_P}=R$ then $(\xi_{V_P}+\hsr)\vert_{\sa_P^R} \in -\sa_P^{*R+}$, a contradiction.

We now use induction on $\#\D_P^R\ge 1$.  In the $\#\D_P^R=1$ case, the note above shows $Q^{\prime R}_{V_P}=P$.  This case follows if we set $V_{\widetilde P}\equiv V_P$.

If $\#\D_P^R> 1$ consider the long exact sequence \eqref{eqnCompareLocalCohomology} with $P\le Q\le Q'$ replaced by $P\le P\le Q^{'R}_{V_P}$:
\begin{equation}\label{eqnSSCosupportLES}
  \dots\to H^{i-1}(i_P^*\jhat_{P*}\jhat_P^*\ihat_{Q^{\prime R}_{V_P}}^!\M)_{V_P} \to
  H^{i}(i_P^!\M)_{V_P}
  \overset{f}\to  H^{i}(i_P^*\ihat_{Q^{\prime R}_{V_P}}^!\M)_{V_P} \to \cdots \ .
\end{equation}
Set $i=[V_R:V_P]+k$.  The middle term is nonzero by assumption.  If $f\neq 0$
then $V_P \in \muwmS(\ihat_R^!\M)$.
%
%
If however $f=0$ then the first term of \eqref{eqnSSCosupportLES} is nonzero.  The Fary spectral sequence \eqref{eqnFarySpectralSequence} abutting to this term 
implies there exists $P<P_1< R$ such that
$H^{[V_{R}:V_P]+k-1}(\n_P^{P_1}; H(i_{P_1}^! \M))_{V_P} \neq0$.  Thus
$H^{[V_R:V_{P_1}]+k-1}(i_{P_1}^! \M)_{V_{P_1}}\neq0$ for some $V_{P_1}$ satisfying
$V_R \succ_{+,\str} V_{P_1}$ by Lemma ~\ref{lemProjectPartialOrder}\ref{itemProjectionOfPartialOrder}.  Since $1\le \#\D_{P_1}^R < \#\D_P^R$ we are done by induction.

The proof of \ref{itemRightFails} is  similar.  We show that $(P,Q^{R}_{V_P})\cap R < R$ and then use induction.  Instead of \eqref{eqnSSCosupportLES} we use the long exact sequence \eqref{eqnCompareLocalCohomology} applied to $\ihat_R^*\M$
  \begin{equation}\label{eqnSSSupportLESsummary}
    \dots\to H^{i}(i_P^*\ihat_{Q^{R}_{V_P}}^!\ihat_R^*\M)_{V_P} \overset{g}\to   H^{i}(i_P^*\M)_{V_P}
    \to H^{i}(i_P^*\jhat_{Q^R_{V_P}*}\jhat_{Q^R_{V_P}}^*\ihat_R^*\M)_{V_P} \to  \cdots\ .
  \end{equation}
  The middle term is nonzero and the critical case is when $g=0$.  This implies that the last term is nonzero which we analyze by the Mayer-Vietoris as opposed to the Fary spectral sequence.  The result is that $H^{[V_R:V_{P_1}]+\#\D_{P_1}^R+k+1}(i_{P_1}^! \M)_{V_{\P_1}}\neq0$ for some $V_{P_1}$ satisfying $V_{P_1}\prec_{-,\str} V_R$ by Lemma ~\ref{lemProjectPartialOrder}\ref{itemProjectionOfPartialOrder}.  Since $\#\D_{P_1}^R < \#\D_P^R$ we are done by induction.
\end{proof}

\section{Eliminating micro-support}
\label{sectEliminatingMicroSupport}

Let $W\subseteq \Xhat$ is an admissible subset with a unique maximal stratum $X_S$.  Consider $\M\in\Mod(\L_W)$ and a middle weight profile $\eta$.  Suppose $V_R\in\etawmS(\M)$ has $\type_{\eta,V_R}(\M)\neq0$ in some degree $d$.  We will show that if $V_R$ is maximal respect to $\LprecLeft$ then there exists a morphism $\phi\colon  \ihat_{R!}\W^\mu\C(V_R)[-d]\to \M$ in the homotopy category which induces a nonzero map on $\type_{\mu,V_R}^d$.  Likewise if $V_R$ is minimal with respect to $\LprecRight$ there exists a morphism $\psi\colon \M \to \ihat_{R*}\W^\nu\C(V_R)[-d]$ which induces a nonzero map on $\type_{\nu,V_R}^d$.  As a result the mapping cone of these morphisms have  $\type_{\eta,V_R}^d$ strictly smaller than that of $\M$.

Note that the need to pass to the homotopy category is due to Proposition ~\ref{propMappingWC}.

\begin{thm}\label{thmRemoveExtremalMicroSupport}
  Let $\M\in\K(\L_W)$, $R\in\Pl(W)$, and $V_R\in\IrrRep(L_R)$.
  \begin{enumerate}[leftmargin=0em, labelsep=.5em, itemindent=2em]

  \item\label{itemMaxImpliesLeftMap}
    Assume $V_R$ is  maximal with respect to $\LprecLeft$ on $\muwmS(\M)$ and $d\in\ZZ$ is such that $\type_{\mu,V_R}^d(\M)\neq0$.  Then there exists a morphism
    \begin{equation*}
      \phi\colon \ihat_{R!}\W^\mu\C(V_R)[-d]\to \M
    \end{equation*}
    for which
    \[
    \type_{\mu,\widetilde V}(\cone(\phi)) =\begin{cases}
    \type_{\mu,V_R}(\M)/V_R[-d] & \text{if $\widetilde V = V_R$ and}\\
    \type_{\mu,\widetilde V}(\M) & \text{if $\widetilde V\neq V_R$}.
    \end{cases}
    \]
\item\label{itemMinImpliesRightMap}
  Assume $V_R$ is minimal with respect to $\LprecRight$ on $\nuwmS(\M)$ and $d\in\ZZ$ is such that $\type_{\nu,V_R}^d(\M)\neq0$.  Then there exists a morphism
  \begin{equation*}
    \psi\colon \M \to \ihat_{R*}\W^\nu\C(V_R)[-d]
  \end{equation*}
  for which
    \[
    \type_{\nu,\widetilde V}(\cone(\phi)) =\begin{cases}
    \type_{\nu,V_R}(\M)\ominus V_R[-d] & \text{if $\widetilde V = V_R$ and}\\
    \type_{\nu,\widetilde V}(\M) & \text{if $\widetilde V\neq V_R$}.
    \end{cases}
    \]
  \end{enumerate}
\end{thm}

\begin{proof}
  To prove \ref{itemMaxImpliesLeftMap} we first demonstrate that the natural morphism $\LeftNaturalMorphism$ from \eqref{eqnNaturalMorphisms} induces a nonzero map on the $V_R$-isotypical part of cohomology:
  \begin{equation}
H^d(i_R^!\M)_{V_R} \xrightarrow{H(\LeftNaturalMorphism)} \type_{\mu,V_R}^d(\M) \ .
  \end{equation}
  For if this were not true then Lemma~ \ref{lemBased}\ref{itemBasedLeftFails} implies that $V_R\prec_- V_{\widetilde R}$ and hence  $V_R \LprecLeft V_{\widetilde R}$ for some $V_{\widetilde R}\in\muwmS(\M)$, contradicting the maximality of $V_R$.

Secondly we note that there exists $\phi_R\colon V_R[-d]\to i_R^!\M$ such that
  \begin{equation}\label{eqnAlphaMap}
V_R[-d] \xrightarrow{\;H(\phi_R)\;}  H^d(i_R^!\M)_{V_R}\xrightarrow{\;\;H(\LeftNaturalMorphism)\;\;} \type_{\mu,V_R}^d(\M)
  \end{equation}
  is nonzero.  This is clear since $H(\LeftNaturalMorphism)\neq0$ implies one can find a copy of $V_R$ in $H^d(i_R^!\M)_{V_R}$ which is not contained in $\ker H(\LeftNaturalMorphism)$ and lift it to the kernel of the complex $i_R^!\M$ in degree $d$.

  The third step is to prove \eqref{eqnLocalCoVanishing} is satisfied, that is, 
  \begin{equation}
    \label{eqnLocalCoVanishingInTheorem}
  H^{[V_R:V_P]+d+1}(i_P^!\N)_{V_P}= 0
\end{equation}
for all $V_P\prec V_R$ with $\xi_{V_P}^R\ge \mu^R_P$.  Note this last inequality is equivalent to $V_R\succ_{+,\str} V_P$  by \eqref{eqnGreaterThanMu}.  Thus if \eqref{eqnLocalCoVanishingInTheorem} fails for such a $V_P$ then
by Lemma~\ref{lemLocalNonvanishing}\ref{itemLeftFails} there exists $V_{P_1}\in \wmS(\ihat_R^!\M)$ with $V_R\succ_{+,\str} V_{P_1}$ and by Lemma~\ref{lemMicrosupportOfPullback}\ref{itemMicrosupportOfProperPullback} there exists $V_{\widetilde P}\in \muwmS(\M)$ with $V_{P_1}\preccurlyeq_- V_{\widetilde P}$.  Together this implies $V_R \LprecLeft V_{\widetilde P}$ which contradicts the maximality of $V_R$ in $\muwmS(\M)$.  Hence \eqref{eqnLocalCoVanishingInTheorem} holds.

The final step is to prove that that $\phi_R$ can be extended to a morphism
  \[
  \phi\colon \ihat_{R!}\W^\mu\C(V_R)[-d]\to \M\ .
  \]
This follows from Proposition ~\ref{propMappingWC}\ref{itemLeft} since 
\eqref{eqnLocalCoVanishing} holds by above.

  By Corollary ~\ref{corEtamSofWetaC}, 
  $\type_{\mu,V_R}(\ihat_{R!}\W^\mu\C(V_R)[-d])$ is the irreducible module $V_R[-d]$ and $\type_{\mu,\widetilde V}(\ihat_{R!}\W^\mu\C(V_R)[-d])=0$ for $\widetilde V\neq V_R$.  The  long exact sequence of $Q$-type \eqref{eqnLESType} then yields the short exact sequences
  \[
  0 \longrightarrow V_R[-d] \longrightarrow \type_{\mu,V_R}(\M) \longrightarrow \type_{\mu,V_R}(\cone(\phi)) \to 0
  \]
  and
  \[
  0 \longrightarrow \type_{\mu,\widetilde V}(\M) \longrightarrow \type_{\mu,\widetilde V}(\cone(\phi)) \to 0 \qquad (\widetilde V\neq V_R)
  \]
from which the theorem follows.  

The proof of \ref{itemMinImpliesRightMap} is similar.
\end{proof}

\section{Bounded $\L$-modules are mixed}
\label{sectBoundedLmodulesMixed}
In this section $\eta$ is a middle weight profile and $W\subseteq \Xhat$ is an admissible subset with unique maximal stratum $X_S$.

\begin{defn}
  \label{defnMixed}
  An $\L$-module $\M$ is \emph{$\mu$-mixed} if two conditions hold.  First there exists
  \begin{enumerate}[label=(\alph*),leftmargin=0em, labelsep=.5em, itemindent=2em]
  \item $\L$-modules $0=\M_0$, $\M_{1},\dots,\M_N=\M$,
  \item modules $V_1,V_{2},\dots,V_N$ where $V_i\in\IrrRep(L_{P_i})$ for some $P_i\in\Pl(W)$, and
  \item degrees $d_1,d_{2},\dots,d_N$.
  \end{enumerate}
Second the $\L$-module $\M_{i-1}=\cone(\phi_i)$ where $\phi_i\colon \ihat_{P_i!}\W^\mu\C(V_i)[-d_i]\to \M_i$ for all $i=1,2,\dots,N$.

  An $\L$-module $\M$ is \emph{$\nu$-mixed}
if similar data exists but for all $i$, $\M_{i-1}=\cone(\psi_i)[-1]$ where $\psi_i\colon \M_i \to \ihat_{P_i*}\W^\nu\C(V_i)[-d_i]$.
\end{defn}

\begin{thm}
  \label{thmBoundedLModulesAreMixed}
  A bounded $\L$-module $\M\in \K^b(\L_W)$ is $\mu$-mixed and $\nu$-mixed.
\end{thm}

Before proving the theorem we need two simple lemmas.

\begin{lem}\label{lemEmptySS}
If $\etawmS(\M)=\emptyset$ then $\M=0$ in the homotopy category.
\end{lem}

\begin{proof}
Assume $\M\neq0$ in the homotopy category.  By Proposition ~\ref{propQuasiIsomorphismIsHomotopyIsomorphism} there exists an $L_P$-module $V\in\IrrRep(\L_W)$ such that $H(i_P^*\M)_{V}\neq0$.  Then $V$ belongs to both $\wmS(\ihat_P^!\M)$ and $\wmS(\ihat_P^*\M)$ so by Lemma ~\ref{lemMicrosupportOfPullback} there exists an element of $\etawmS(\M)$, a contradiction.
\end{proof}

\begin{lem}\label{lemBoundedMeansWmSFinite}
  Assume $\M$ is bounded.  Then $\wmS(\M)$ is finite.  Furthermore if $V\in\wmS(\M)$ then $\type_{Q,V}(\M)$ is finite-dimensional for $Q\in[Q_V^W,Q_V^{\prime W}]$.
\end{lem}

\begin{proof}
  If $\M$ is bounded then for each $P\in\Pl(W)$ the direct sum $\bigoplus_i E_P^i$ is regular and hence finite-dimensional.  Thus it has only finitely many nonzero isotypical components.
\end{proof}

\begin{proof}[Proof of Theorem~ \textup{\ref{thmBoundedLModulesAreMixed}}]
  Since $\M$ is bounded, $\etawmS(\M)$ is finite and $\type_\eta(\M) \equiv \bigoplus_{V\in\etawmS(\M)}  \type_{\eta,V}(\M)$ is finite-dimensional as noted in Lemma~\ref{lemBoundedMeansWmSFinite}.  We use induction on $\dim \type_\eta(\M)$.  If the dimension is $0$ then $\etawmS(\M)=\emptyset$ and Lemma~\ref{lemEmptySS} finishes the proof.  If the dimension is $>0$ then $\etawmS(\M)\neq\emptyset$.  Choose an $L_R$-module $V\in \etawmS(\M)$ to be maximal with respect to $\LprecLeft$ (when $\eta=\mu$) or minimal with respect to $\LprecRight$ (when $\eta= \nu$).  In the $\eta=\mu$ case Theorem ~\ref{thmRemoveExtremalMicroSupport}\ref{itemMaxImpliesLeftMap} shows there exists a morphism $\phi\colon \ihat_{R!}\W^\mu\C(V_R)[-d]\to \M$ such that the multiplicity of $V_R$ in $\type_\mu(\cone(\phi))$ is one less than in $\type_\mu(\M)$ and the multiplicities of $\widetilde V\neq V_R$ are unchanged.  Set $\widetilde M= \cone(\phi)$.  Then $\dim \type_\eta(\widetilde \M) < \dim \type_\eta(\M)$ and we are done by induction.   In the $\eta=\nu$ case the same argument holds except with a morphism $\psi\colon \M \to \ihat_{R*}\W^\nu\C(V_R)[-d]$.
\end{proof}

From the proof above we obtain the
\begin{cor}\label{corDataAboutBuildingBlocks}
The $\eta$-mixed data of $\M$ contains all $V_i\in\etawmS(\M)$ with multiplicity equal to the multiplicity of $V_i$ in  $\type_{\eta,V_i}(\M)$.  As $i$ goes from $1$ to $N$, the $V_i$ are nondecreasing with respect to $\LprecLeft$ \textup(for $\eta=\mu$\textup) or nonincreasing with respect to $\LprecRight$ \textup(for $\eta=\nu$\textup).
\end{cor}

\section{Intersection cohomology equals weighted cohomology}
We recall the construction of the intersection cohomology $\L$-module $\I_p\C(E)$ for $E\in\IrrRep(G)$ \cite[\S5]{refnSaperLModules}.  Its realization is the Deligne sheaf for Goresky and MacPherson's intersection cohomology \cite[\S3]{refnGoreskyMacPhersonIHTwo}, \cite[V, \S2]{refnBorelIntersectionCohomology}, a topological invariant.  (We use cohomological indexing as opposed to the perverse indexing from \cite{refnBeilinsonBernsteinDeligne}.)

For a middle perversity $m$ or $n$ we prove (under a condition on the $\QQ$-root system) that if the coefficient system $\EE$ arises from a conjugate-self contragradient $G$-module then global intersection cohomology for $m$ and $n$ is isomorphic to global weighted cohomology for weight profiles $\mu$ and $\nu$ respectively.

\subsection{Intersection cohomology as an $\L$-module}
Given an $\L$-module $\M$ its \emph{degree truncation} along the $X_Q$ stratum is the mapping cone
\[
\tau_Q^{\le n}\M = \cone(\M\to i_{Q*}\tau^{>n}i_Q^*\M)[-1]
\]
where $\tau^{>n}$ is the usual truncation of a complex.

Let $p\colon\{2,\dots,\dim X\}\to \ZZ$ be a classical perversity.  The $\L$-module $\I_p\C(E)$ is
\[
\I_p\C(E) =  \t_{Q_1}^{\le p(\codim_{\Xhat} X_{Q_1})}j_{Q_1*}\cdots\t_{Q_N}^{\le p(\codim_{\Xhat}X_{Q_N})}j_{Q_N*}i_{G*}E
\]
where $Q_1,\dots,Q_N$ is an enumeration of $\Pl(\Xhat)\setminus\{G\}$  such that if $Q_i<Q_j$ then $i<j$.  The realization of $\I_p\C(E)$ is isomorphic in the derived category to Deligne's intersection cohomology sheaf $\I_p\C_p(\Xhat;\EE)$.

\subsection{The micro-support of intersection cohomology}
The lower and upper middle perversities are defined by $m(k)=\lfloor \frac{k-2}2 \rfloor$ and $n(k)=\lfloor\frac{k-1}2 \rfloor$.  For a middle peversity $p$ the micro-support of the intersection cohomology $\L$-module was calculated in \cite[Theorem~ 17.1, Corollary 17.2]{refnSaperLModules} (see also \cite[Lemma ~8.8]{refnSaperLModules}):

\begin{prop}
  \label{propMicroSupportIntersectionCohomology}
  Let $E\in \IrrRep(G)$ satisfy the conjugate self-contragradient condition $(E|_{M_G})^*\cong \overline{E|_{M_G}}$.  Assume the $\QQ$-root system of $G$ does not involve types $D$, $E$, and $F$.  Let $p$ be a middle perversity.  Then $V_P\in  \mS(\I^p\C(E))$ if and only if $V_P \preccurlyeq_0 E$ and $V_P|_{M_P}$ is conjugate self-contragradient.  Furthermore\textup:
  \begin{enumerate}[leftmargin=0em, labelsep=.5em, itemindent=2em]
  \item For such $V_P$,
    \begin{equation}\label{eqnmSForHmC}
    H(i_P^*\ihat_Q^!(\I^m\C(E))_{V_P} = V_P[-\frac12\dim \n_P-\#\D_P]
    \end{equation}
    when $Q=Q_{V_P}=P$ and is zero otherwise.
  \item For such $V_P$,
    \begin{equation}\label{eqnmSForHnC}
    H(i_P^*\ihat_Q^!(\I^n\C(E))_{V_P} = V_P[-\frac12\dim \n_P]
    \end{equation}
    when $Q=Q_{V_P}'=G$ and is zero otherwise.
  \end{enumerate}
\end{prop}

Since in the above setting $Q_{V_P}=Q_{V_P}'$ if and only if $P=G$ we have the
\begin{cor}
  \label{corMSofIpC}
For middle perversity intersection cohomlogy,
  \[
  \mS_\mu(\I^m\C(E))= \{E\} \quad\text{and}\quad  \mS_\nu(\I^n\C(E))= \{E\}\ .
  \]
Furthermore $\type_{\mu,E}(\I^m\C(E)) = E$ and $\type_{\nu,E}(\I^n\C(E)) = E$.
\end{cor}

\begin{rem*}
  See \cite[\S17]{refnSaperLModules} for the more complicated description of  $\mS(\I^p\C(E))$ without the assumption $(E|_{M_G})^*\cong \overline{E|_{M_G}}$.  Also we do not have a description in general of the weak micro-support for $\I^p\C(E)$; its elements are not determined simply by dropping the conjugate self-contragradient condition on $V_P|_{M_P}$ in the proposition.
\end{rem*}

\subsection{Isomorphism of intersection cohomology and weighted cohomology}

Let $I_p H(\Xhat;\EE)= H(\Xhat;\I_p\C(E))$ be the global intersection cohomology for perversity $p$.

\begin{thm}
  \label{thmIsomorphismIntersectionAndWeightedCohomology}
    Let $E\in \IrrRep(G)$ satisfy the conjugate self-contragradient condition $(E|_{M_G})^*\cong \overline{E|_{M_G}}$.  Assume the $\QQ$-root system of $G$ does not involve types $D$, $E$, and $F$.  Then
  \[
  W^\mu H(\Xhat;\EE) \cong I_mH(\Xhat;\EE)\quad \text{and} \quad
  W^\nu H(\Xhat;\EE) \cong I_nH(\Xhat;\EE)\ .
  \]
\end{thm}

\begin{proof}
We only give the proof for $\mu$.  By Theorem ~\ref{thmBoundedLModulesAreMixed} the $\L$-module $\I^m\C(E)$ is $\mu$-mixed.  Let $(\M_i)_{i=0}^N$ be the sequence of $\L$-modules from Definition ~\ref{defnMixed} with $\M_0=0$ and $\M_N=\I^m\C(E)$.  There are distinguished triangles $\to\W^\mu\C(V_i)[-d_i]\to \M_i \to \M_{i-1}\to$ where $V_i$ ranges over the elements of $\muwmS(\I^m\C(E))$.  By Corollary ~\ref{corDataAboutBuildingBlocks} and Corollary ~\ref{corMSofIpC}, there is a unique index $i_0$ with $V_{i_0}=E\in \mS_\mu(\I^m\C(E))$ and all other $V_i\in \muwmS(\I^m\C(E))\setminus \{E\}$.
By Corollary ~\ref{corWeightedCohomologyVanishing}  $W^\mu H(\Xhat_{P_i};\VV_i) = 0$ for all $i\neq i_0$.  The long exact sequences of the above distinguished triangles then imply $H(\Xhat;\M_i)=0$ for $i<i_0$, $H(\Xhat;\M_{i_0})\cong W^\mu H(\Xhat;\EE)$, and $H(\Xhat;\M_i) \cong H(\Xhat;\M_{i-1})$ for $i>i_0$.  The theorem follows since $H(\Xhat;\M_N)= I_mH(\Xhat;\EE)$.
\end{proof}

In particular the theorem shows that weighted cohomology is a topological invariant  in these cases.


\bibliographystyle{amsplain}
\bibliography{../references}
\end{document}